\documentclass[12pt]{amsart}

\usepackage{etex, amsfonts, amsmath, amsthm, amssymb, stmaryrd, epsfig, graphics,
  psfrag, latexsym, color, mathtools, mathrsfs,enumitem, subfig,
  longtable, booktabs, yfonts} \usepackage[all,cmtip]{xy}
\usepackage[percent]{overpic}
\definecolor{darkblue}{rgb}{0,0,0.4} \usepackage[colorlinks=true,
citecolor=darkblue, filecolor=darkblue, linkcolor=darkblue,
urlcolor=darkblue]{hyperref} \usepackage[all]{hypcap}
\usepackage{xr-hyper}

\graphicspath{{draws/}{}}

\numberwithin{equation}{section}
\newtheorem{thm}{Theorem}
\newtheorem{theorem}[thm]{Theorem}
\newtheorem{thmcor}[thm]{Corollary}
\newtheorem{lem}{Lemma}[section]               
\newtheorem{lemma}[lem]{Lemma}               

\newtheorem{cor}[lem]{Corollary}
\newtheorem{corollary}[lem]{Corollary}               
\newtheorem{prop}[lem]{Proposition}

\newtheorem{citethm}[lem]{Theorem}
\theoremstyle{definition}     

\newtheorem{question}[lem]{Question}
\newtheorem{defn}[lem]{Definition} 
\newtheorem{definition}[lem]{Definition}

\theoremstyle{remark}     
\newtheorem{rem}{Remark}[section]
\newtheorem{remark}[rem]{Remark}

\numberwithin{figure}{section}

\newcommand{\Appendix}[1]{\hyperref[app:#1]{Appendix~\ref*{app:#1}}}
\newcommand{\Section}[1]{\hyperref[sec:#1]{Section~\ref*{sec:#1}}}
\newcommand{\Subsection}[1]{\hyperref[subsec:#1]{Subsection~\ref*{subsec:#1}}}
\newcommand{\Lemma}[1]{\hyperref[lem:#1]{Lemma~\ref*{lem:#1}}}
\newcommand{\Theorem}[1]{\hyperref[thm:#1]{Theorem~\ref*{thm:#1}}}
\newcommand{\ThmCor}[1]{\hyperref[thmcor:#1]{Corollary~\ref*{thmcor:#1}}}
\newcommand{\Citethm}[1]{\hyperref[citethm:#1]{Theorem~\ref*{citethm:#1}}}
\newcommand{\Definition}[1]{\hyperref[def:#1]{Definition~\ref*{def:#1}}}
\newcommand{\Remark}[1]{\hyperref[rem:#1]{Remark~\ref*{rem:#1}}}
\newcommand{\Figure}[1]{\hyperref[fig:#1]{Figure~\ref*{fig:#1}}}
\newcommand{\Conjecture}[1]{\hyperref[conj:#1]{Conjecture~\ref*{conj:#1}}}
\newcommand{\Corollary}[1]{\hyperref[cor:#1]{Corollary~\ref*{cor:#1}}}
\newcommand{\Proposition}[1]{\hyperref[prop:#1]{Proposition~\ref*{prop:#1}}}
\newcommand{\Question}[1]{\hyperref[ques:#1]{Question~\ref*{ques:#1}}}
\newcommand{\Example}[1]{\hyperref[exam:#1]{Example~\ref*{exam:#1}}}
\newcommand{\Table}[1]{\hyperref[table:#1]{Table~\ref*{table:#1}}}
\newcommand{\Restric}[1]{\hyperref[restric:#1]{Restriction~\ref*{restric:#1}}}

\newcommand{\Equation}[2][{}]{Equation#1~(\ref{eq:#2})}

\newcommand{\R}{\mathbb{R}}
\newcommand{\Z}{\mathbb{Z}}
\newcommand{\F}{\mathbb{F}}

\newcommand{\mc}{\mathcal}

\newcommand{\bm}{\mathbf}

\newcommand{\wh}{\widehat}
\newcommand{\wt}{\widetilde}
\newcommand{\ol}{\overline}

\newcommand{\sbs}{\subset}

\newcommand{\al}{\alpha}
\newcommand{\be}{\beta}

\newcommand{\ep}{\epsilon}

\newcommand{\from}{\colon}

\newcommand{\set}[2]{\{#1\mid#2\}}
\newcommand{\restrict}[2]{{#1}|_{#2}}

\renewcommand{\th}{^{\text{th}}}

\DeclareMathOperator{\Id}{Id}
\DeclareMathOperator{\Sq}{Sq}

\newcommand{\Kh}{\mathit{Kh}}

\newcommand{\KhCx}{\mathit{KC}}

\newcommand{\Realize}[2][{}]{|#2|_{#1}}

\newcommand{\CubeFlowCat}{\mathscr{C}_C}
\newcommand{\KhFlowCat}{\mathscr{C}_K}

\newcommand{\Moduli}{\mathcal{M}}

\newcommand{\gr}{\mathrm{gr}}

\newcommand{\KhSpace}{\mathcal{X}_\mathit{Kh}}

\newcommand{\TupV}{\mathbf}
\newcommand{\Frame}{\varphi}
\newcommand{\Cube}{\mathcal{C}}

\newcommand{\diff}{\delta}

\newcommand{\co}{\colon}
\newcommand{\bdy}{\partial}

\newcommand{\ZZ}{\mathbb{Z}}

\newcommand{\respectively}{resp.\ }

\newcommand*{\defeq}{\mathrel{\vcenter{\baselineskip0.5ex \lineskiplimit0pt
                     \hbox{\scriptsize.}\hbox{\scriptsize.}}}%
                     =}

\newcommand{\BNcx}{C}
\newcommand{\Filt}{\mathcal{F}}
\newcommand{\Red}{\textcolor{red}}
\newcommand{\xbar}{x_{\text{\rotatebox[origin=c]{90}{$-$}}}}
\newcommand{\Field}{\mathbb{F}}
\DeclareMathOperator{\characteristic}{\mathit{char}}
\newcommand{\QQ}{\mathbb{Q}}
\newcommand{\CubeCx}{C^*_{\mathit{cube}}}
\newcommand{\signass}{\mathfrak{s}}
\newcommand{\tsignass}{\mathfrak{t}}

\begin{document}
\title{A refinement of Rasmussen's s-invariant}

\author{Robert Lipshitz}
\thanks{RL was supported by NSF grant number DMS-0905796 and a Sloan Research Fellowship.}
\email{\href{mailto:lipshitz@math.columbia.edu}{lipshitz@math.columbia.edu}}

\author{Sucharit Sarkar}
\thanks{SS was supported by a Clay Mathematics Institute Postdoctoral Fellowship}
\email{\href{mailto:sucharit@math.columbia.edu}{sucharit@math.columbia.edu}}

\subjclass[2010]{\href{http://www.ams.org/mathscinet/search/mscdoc.html?code=57M25,55P42}{57M25,
    55P42}}

\address{Department of Mathematics, Columbia University, New York, NY 10027}

\date{\today}

\begin{abstract}
  In~\cite{RS-khovanov} we constructed a spectrum-level refinement of
  Khovanov homology. This refinement induces stable cohomology
  operations on Khovanov homology. In this paper we show that these
  cohomology operations commute with cobordism maps on Khovanov
  homology. As a consequence we obtain a refinement of Rasmussen's
  slice genus bound $s$ for each stable cohomology operation. We show
  that in the case of the Steenrod square $\Sq^2$ our refinement is
  strictly stronger than $s$.
\end{abstract}

\maketitle

\tableofcontents

\section{Introduction}
In~\cite{RS-khovanov} we gave a space-level refinement of Khovanov
homology. That is, given a link $L$ we produced a family of suspension
spectra $\KhSpace^j(L)$ with the property that
$\wt{H}^i(\KhSpace^j(L))=\Kh^{i,j}(L)$.  In this paper, we use these
Khovanov spectra to give a family of potential improvements of
Rasmussen's celebrated $s$ invariant~\cite{Ras-kh-slice}; and we
show that at least one of these is, in fact, an improvement on $s$.

These refinements are fairly easy to state. To wit, let
$\Kh^{i,j}(L;\Field)$ denote Khovanov homology with coefficients in a
field $\Field$. There is a spectral sequence
$\Kh^{i,j}(L;\Field)\rightrightarrows \Field^{2^{|L|}}$, coming from a
filtered chain complex $(\BNcx_*,\Filt_\bullet)$. (Here,
$\Filt_\bullet$ is a descending filtration, with
$\Kh^{*,j}(L;\Field)=H_*(\Filt_j\BNcx/\Filt_{j+2}\BNcx)$.)
Originally, this was defined by Lee~\cite{Lee-kh-endomorphism}, for
fields $\Field$ of characteristic different from $2$. A variant which
works for all fields (in fact, all rings) was studied by
Bar-Natan~\cite{Bar-kh-tangle-cob} and
Turner~\cite{Turner-kh-BNseq}. We will work with this variant, which
is reviewed in \Section{BN-cx}.

The Rasmussen $s$ invariant for a knot $K$ is defined by
\begin{align*}
s^\Field(K)&=\max\set{q\in2\Z+1}{i_*\co H_*(\Filt_q\BNcx)\to H_*(\BNcx)\cong \Field^2\text{
  surjective}} + 1 \\
&= \max\set{q\in 2\Z+1}{i_*\co H_*(\Filt_q\BNcx)\to H_*(\BNcx)\cong \Field^2\text{
  nonzero}} - 1.\footnotemark
\end{align*}
and gives a lower bound for the slice genus: $|s^\Field(K)|\leq 2g_4(K)$.  It
is shown in~\cite{MTV-Kh-s-invts} that if $\characteristic(\Field)\neq
2$ then it makes no difference whether
one uses the Bar-Natan deformation or the Lee deformation in the
definition of $s$; see the discussion around \Citethm{MTV}, below.
\footnotetext{For justification of this equality in the
  case of fields of characteristic $2$, see
  \Proposition{differ-by-two}, below. Note that, while it is claimed
  in~\cite{MTV-Kh-s-invts} that $s^\Field$ is independent of $\Field$,
  there is a gap in the proof of~\cite[Proposition
  3.2]{MTV-Kh-s-invts}.}

To define the improvements, let $\alpha\co
\wt{H}^*(\cdot; \Field)\to \wt{H}^{*+n}(\cdot; \Field)$ be a stable cohomology
operation (for some $n>0$).
\begin{defn}\label{def:full}
  Fix a knot $K$.  Call an odd integer $q$ \emph{$\alpha$-half-full} if
  there exist elements $\wt{a}\in\Kh^{-n,q}(K;\Field)$,
  $\wh{a}\in\Kh^{0,q}(K;\Field)$, $a\in H_0(\Filt_q;\Field)$
  and $\ol{a}\in H_0(\BNcx,\Field)$ satisfying:
\begin{enumerate}
\item the map $\alpha\from\Kh^{-n,q}(K;\Field)=\wt{H}^{-n}(\KhSpace^q;\Field)\to
  \wt{H}^{0}(\KhSpace^q;\Field)=\Kh^{0,q}(K;\Field)$ sends $\wt{a}$ to
  $\wh{a}$;
\item the map
  $H_0(\Filt_q;\Field)\to\Kh^{0,q}(K;\Field)=H_0(\Filt_q/\Filt_{q+2};\Field)$ sends
  $a$ to $\wh{a}$;
\item the map $H_0(\Filt_q;\Field)\to H_0(\BNcx;\Field)$ sends $a$ to
$\ol{a}$; and
\item $\ol{a}\in H_0(\BNcx;\Field)=\Field\oplus \Field$ is a generator.
\end{enumerate}
(Note that $\wt{a}$ and $\wh{a}$ are allowed to be zero.)

Call an odd integer $q$ \emph{$\alpha$-full} if there exist elements
$\wt{a},\wt{b}\in\Kh^{-n,q}(K;\Field)$, $\wh{a},\wh{b}\in\Kh^{0,q}(K;\Field)$, $a,b\in
H_0(\Filt_q;\Field)$ and $\ol{a},\ol{b}\in H_0(\BNcx;\Field)$ satisfying:
\begin{enumerate}
\item the map
  $\alpha\from\Kh^{-n,q}(K;\Field)=\wt{H}^{-n}(\KhSpace^q;\Field)\to
  \wt{H}^{0}(\KhSpace^q;\Field)=\Kh^{0,q}(K;\Field)$ sends
  $\wt{a},\wt{b}$ to $\wh{a},\wh{b}$;
\item the map $H_0(\Filt_q;\Field)\to\Kh^{0,q}(K;\Field)=H_0(\Filt_q/\Filt_{q+2};\Field)$ sends $a,b$ to
$\wh{a},\wh{b}$;
\item the map $H_0(\Filt_q;\Field)\to H_0(\BNcx;\Field)$ sends $a,b$ to
$\ol{a},\ol{b}$; and
\item $\ol{a},\ol{b}\in H_0(\BNcx;\Field)=\Field\oplus \Field$ form a basis.
\end{enumerate}
(Again, note that $\wt{a}$, $\wt{b}$, $\wh{a}$ and $\wh{b}$ are
allowed to be zero.)

In other words, $q$ is \emph{$\alpha$-half-full} if the following
configuration exists:
\[
\xymatrix{
\langle\wt{a}\rangle\ar[r]\ar@{^(->}[d]&\langle\wh{a}\rangle\ar@{<-}[r]\ar@{^(->}[d]&\langle
a\rangle\ar[r]\ar@{^(->}[d]&\langle\ol{a}\rangle\neq 0\ar@<-2ex>@{^(->}[d]\\
\Kh^{-n,q}(K;\Field)\ar[r]^-{\alpha}&\Kh^{0,q}(K;\Field)\ar@{<-}[r]& H_0(\Filt_q;\Field)\ar[r]& H_0(\BNcx;\Field). 
}               
\]
while $q$ is \emph{$\alpha$-full} if the following configuration exists:
\[
\xymatrix{
\langle\wt{a},\wt{b}\rangle\ar[r]\ar@{^(->}[d]&\langle\wh{a},\wh{b}\rangle\ar@{<-}[r]\ar@{^(->}[d]&\langle
a,b\rangle\ar[r]\ar@{^(->}[d]&\langle\ol{a},\ol{b}\ar@{=}[d]\rangle\\
\Kh^{-n,q}(K;\Field)\ar[r]^-{\alpha}&\Kh^{0,q}(K;\Field)\ar@{<-}[r]& H_0(\Filt_q;\Field)\ar[r]& H_0(\BNcx;\Field). 
}               
\]
\end{defn}

\begin{defn}\label{def:rpm-spm}
  For a knot $K$, define $r_+^\alpha(K)=\max\set{q\in 2\Z+1}{q\text{ is
$\alpha$-half-full}}+1$ and $s_+^\alpha(K)=\max\set{q\in 2\Z+1}{q\text{ is
$\alpha$-full}}+3$. If $\ol{K}$ denotes the mirror of $K$, define
$r_-^\alpha(K)=-r_+^\alpha(\ol{K})$ and $s_-^\alpha(K)=-s_+^\alpha(\ol{K})$.
\end{defn}

It is immediate from their definitions that $s_\pm^\alpha$ and
$r_\pm^\alpha$ are knot invariants.  The reason they are of interest
is the following:

\begin{theorem}\label{thm:slice-bound}
  Let $\alpha$ be a stable cohomology operation and $S$ a connected,
  embedded cobordism in $[0,1]\times S^3$ from $K_1$ to $K_2$. If $S$
  has genus $g$ then
  \begin{align*}
  |s^{\al}_+(K_1)-s^{\al}_+(K_2)|&\leq 2g &   |s^{\al}_-(K_1)-s^{\al}_-(K_2)|&\leq 2g\\
  |r^{\al}_+(K_1)-r^{\al}_+(K_2)|&\leq 2g &   |r^{\al}_-(K_1)-r^{\al}_-(K_2)|&\leq 2g.  
  \end{align*}
  So, each of the number
  $|r_+^\alpha|$, $|r_-^\alpha|$, $|s_+^\alpha|$, $|s_-^\alpha|$
  gives a slice genus bound:
  \[
  \max\{|r_+^\alpha(K)|,|r_-^\alpha(K)|,|s_+^\alpha(K)|,|s_-^\alpha(K)|\}\leq 2g_4(K).
  \]
\end{theorem}

(Note that if $s_\pm^\alpha$ (respectively $r_\pm^\alpha$) differs
from $s$, then $s_\pm^\alpha$ is not a slice \emph{homomorphism}, as two
homomorphisms to $\ZZ$ cannot have a bounded difference.)

Of course, the numbers $r_\pm^\alpha$ and $s_\pm^\alpha$ are only
interesting if they sometimes give better information than
$s$. In \Section{computations} we show:
\begin{theorem}\label{thm:sq2-good}
  Let $\Sq^2\co \wt{H}^*(\cdot; \F_2)\to \wt{H}^{*+2}(\cdot; \F_2)$
  denote the second Steenrod square. Then there are knots $K$ so that
  $|s_+^{\Sq^2}(K)| > |s(K)|$.
\end{theorem}

\begin{theorem}\label{thm:sq1-good}
  Let $\Sq^1\co \wt{H}^*(\cdot; \F_2)\to \wt{H}^{*+1}(\cdot; \F_2)$
  denote the first Steenrod square. Then there are knots $K$ so that
  $s_+^{\Sq^1}(K) \neq s^{\Field_2}(K)$.
\end{theorem}

A key step in the proof of \Theorem{slice-bound} is that the cobordism
maps on Khovanov homology commute with cohomology operations:
\begin{theorem}\label{thm:coho-op-commute}
  Let $S$ be a smooth cobordism in $[0,1]\times S^3$ from $L_1$ to
  $L_2$, and let $F_S\co \Kh^{*,*}(L_1)\to \Kh^{*,*+\chi(S)}(L_2)$ be the map
  associated to $S$ in~\cite{Jac-kh-cobordisms} (see
  also~\cite{Kho-kh-tangles, Bar-kh-tangle-cob, Kho-kh-cob}). Let
  $\alpha\co \wt{H}^*(\cdot; \Field)\to \wt{H}^{*+n}(\cdot; \Field)$ be a stable
  cohomology operation. Then the following diagram commutes up to
  sign:
  \begin{equation}\label{eq:coho-op-commute}
  \vcenter{\hbox{
      \xymatrix@C=2.2ex{
        \Kh^{i,j}(L_1;\Field)=\wt{H}^i(\KhSpace^j(L_1);\Field)\ar@<-10ex>[d]_{F_S}
        \ar[r]^-{\alpha}&\wt{H}^{i+n}(\KhSpace^j(L_1);\Field)=\Kh^{i+n,j}(L_1;\Field)
        \ar@<8ex>[d]^{F_S}\\
        \mathllap{\Kh}{}^{i,j+\chi(S)}(L_2;\Field)=\wt{H}^i(\KhSpace^{j+\chi(S)}(L_2);\Field)
        \ar[r]^-{\alpha}&\wt{H}^{i+n}(\KhSpace^{j+\chi(S)}(L_2);\Field)=\Kh^{i+n,j+\chi(S)}(L_2;\Field).
      }
    }}
  \end{equation}
\end{theorem}

In particular:
\begin{thmcor}\label{thmcor:steenrod-alg}
  Let $\mathcal{A}_p$ denote the modulo-$p$ Steenrod algebra. Then the
  cobordism map $F_S\co \Kh^{*,*}(L_1;\F_p)\to
  \Kh^{*,*+\chi(S)}(L_2;\F_p)$ associated to a smooth cobordism $S$
  from $L_1$ to $L_2$ is a homomorphism of $\mathcal{A}_p$-modules;
  that is, $\Kh(\cdot;\F_p)$ is a projective functor of $\mc{A}_p$-modules.
\end{thmcor}

This paper is organized as follows. \Section{BN-cx} reviews the
Bar-Natan complex and the $s$-invariant defined using it, collecting
some results we will need later. \Section{cob-commutes-st} starts by
reviewing the construction of the Khovanov homotopy type. The rest of 
\Section{cob-commutes-st} is devoted
to proving \Theorem{coho-op-commute}. In \Section{new-s} we recall the
definitions of $r_\pm^\alpha$ and $s_\pm^\alpha$ and prove
\Theorem{slice-bound}. \Section{computations} contains some
computations of the invariants $r_\pm^\alpha(K)$ and $s_\pm^\alpha(K)$
for some particular $\alpha$'s and $K$'s, and in particular gives
proofs of \Theorem{sq2-good}
and \Theorem{sq1-good}. We conclude,
in \Section{further}, with some remarks and questions.

\thinspace\thinspace\thinspace
\noindent
\textbf{Acknowledgements.} We thank T.~Jaeger, M.~Khovanov, M.~Mackaay,
P.~Ozsv\'ath, J.~Rasmussen, C.~Seed, P.~Turner, and P.~Vaz for helpful
conversations. We also thank the referees for their time and their
helpful suggestions and corrections.

\section{The \texorpdfstring{$s$}{s} invariant from Bar-Natan's complex}\label{sec:BN-cx}
In this section we review some results on Bar-Natan's filtered
Khovanov complex. For our purposes, it serves as an analogue of the
Lee deformation but which works over any coefficient ring. Almost all
of the ideas and many of the results in this section are drawn
from~\cite{Lee-kh-endomorphism,Ras-kh-slice,Bar-kh-tangle-cob,Turner-kh-BNseq,MTV-Kh-s-invts},
but a few of the results have not appeared in exactly the form we need
them. We start by reviewing the filtered complex itself and its basic
properties in \Subsection{BN-cx-1}, and then turn
in \Subsection{BN-s-invt} to the properties of the $s$-invariants
obtained from the filtered complex.
\subsection{Bar-Natan's complex}\label{subsec:BN-cx-1}
Like the Khovanov complex, the Bar-Natan complex is obtained by
feeding the cube of resolutions for a knot diagram into a particular
Frobenius algebra:
\begin{definition}\label{def:BN-frob-alg}
  The \emph{Bar-Natan Frobenius algebra} is the deformation of
  $H^*(S^2)$ with
  multiplication $m$ given by
  \[
  x_+\otimes x_+\mapsto x_+ \qquad x_+\otimes x_-\mapsto x_-\qquad
  x_-\otimes x_+\mapsto x_- \qquad x_-\otimes x_-\mapsto \Red{x_-},
  \]
  comultiplication $\Delta$ by
  \[
  x_-\mapsto x_-\otimes x_-\qquad x_+\mapsto x_+\otimes x_-+x_-\otimes
  x_+\Red{-x_+\otimes x_+},
  \]
  unit $\iota$ by
  \[
  1\mapsto x_+
  \]
  and counit $\eta$ by
  \[
  x_+\mapsto 0 \qquad x_-\mapsto 1.
  \]
  (These maps are obtained from the usual Khovanov maps
  from~\cite{Kho-kh-categorification} by adding the terms in
  \Red{red}. These maps make sense over any ring; but we will continue
  to assume that we are working over a field $\Field$.)
\end{definition}

Feeding the cube of resolutions for a knot diagram $K$ into this Frobenius algebra gives a
chain complex $\BNcx(K;\Field)$, which we will sometimes call the Bar-Natan
complex. As an $\Field$-module, the underlying chain group is identified
with the Khovanov complex $\KhCx(K;\Field)$.  The Bar-Natan differential
increases homological grading on the Khovanov complex by $1$ and does
not decrease the quantum grading. Write
$\BNcx=\bigoplus_{i,j}\BNcx^{i,j},$ where $i$ denotes the homological
grading and $j$ denotes the quantum grading on the Khovanov
complex. Consider the subcomplex $\Filt_q=\bigoplus_{j\geq
  q}\BNcx^{i,j}$. This gives a (finite) filtration
\[
\cdots\sbs \Filt_{q+2}\sbs \Filt_q\sbs \Filt_{q-2}\sbs\cdots\subset \BNcx.
\]

\begin{citethm}\label{citethm:Turner}
  \cite{Turner-kh-BNseq} If $K$ is a knot then
  $H_*(\BNcx(K;\Field))=\Field\oplus \Field$, with both copies in homological grading
  $0$. More generally, for $L$ an $\ell$-component link,
  $H_*(\BNcx(L;\Field))\cong \Field^{2^{\ell}}$; a basis for
  $H_*(\BNcx(L;\Field))$ is canonically identified with the set of
  orientations for $L$.
\end{citethm}
\begin{proof}[Sketch of proof.]
  This is a special case of~\cite[Propositions 2.3 and 2.4]{MTV-Kh-s-invts}.
  Following~\cite{Turner-kh-BNseq,MTV-Kh-s-invts}, write $\xbar = x_+ - x_-$ and
  consider the new basis $\{\xbar,x_-\}$ for the Bar-Natan Frobenius
  algebra. In this basis, the multiplication, the comultiplication, the
  unit and the counit become
  \begin{equation}\label{eq:change-of-basis}
    \begin{split}
      \xbar\otimes \xbar\stackrel{m}{\longrightarrow}\xbar\qquad
      \xbar\otimes x_-\stackrel{m}{\longrightarrow}0&\qquad
      x_-\otimes \xbar\stackrel{m}{\longrightarrow}0\qquad
      x_-\otimes x_-\stackrel{m}{\longrightarrow} x_-\\
      \xbar\stackrel{\Delta}{\longrightarrow} -\xbar\otimes \xbar&\qquad
      x_-\stackrel{\Delta}{\longrightarrow} x_-\otimes x_-\\
    1\stackrel{\iota}{\longrightarrow}&\,\xbar+x_-\\
    \xbar\stackrel{\eta}{\longrightarrow}-1&\qquad 
    x_-\stackrel{\eta}{\longrightarrow}1.
    \end{split}
  \end{equation}
Since this change of basis diagonalizes the Frobenius algebra, the rest of the argument
  from~\cite{Lee-kh-endomorphism} goes through essentially unchanged.
\end{proof}

The following is essentially due to Bar-Natan:
\begin{citethm}\label{citethm:BN-filt-cob}
  \cite{Bar-kh-tangle-cob} The filtered complex $\BNcx$ is
  projectively functorial
  with respect to link cobordisms, in the sense that given a link
  cobordism $S$ from $L_1$ to $L_2$ there is an associated chain map
  $F_S\co \BNcx(L_1)\to \BNcx(L_2)$, well-defined up to multiplication
  by $\pm 1$. The chain map $F_S$ preserves the
  homological grading and increases the quantum grading by at least
  $\chi(S)$, and is well-defined up to filtered
  homotopy\footnote{i.e., homotopies which decrease the
    homological grading by $1$ and increase the quantum grading by at
    least $\chi(S)$.} (and sign). The map of associated graded complexes induced
  by $F_S$ agrees with the usual cobordism map on Khovanov homology
  (as defined in~\cite{Jac-kh-cobordisms}).
\end{citethm}
\begin{proof}[Sketch of proof.]
  It suffices to show that $\BNcx(L)$ is obtained by composing
  Bar-Natan's formal-complex-valued invariant~\cite[Definition
  6.4]{Bar-kh-tangle-cob} and some functor $\mathcal{C}ob^3_{/l}\to
  \mathcal{K}om_{\ZZ}$.  With coefficients in $\F_2$ instead of $\ZZ$
  this is essentially~\cite[Exercise 9.5]{Bar-kh-tangle-cob}, and is
  discussed further in~\cite{Turner-kh-BNseq}. To obtain the result
  over $\ZZ$, by~\cite[Theorem 5]{Bar-kh-tangle-cob} it suffices
  verify that the topological field theory corresponding to the
  Frobenius algebra in \Definition{BN-frob-alg} satisfies the $S$, $T$
  and $4\mathit{Tu}$ relations.

  The argument is essentially the same as~\cite[Proposition 7.2]{Bar-kh-tangle-cob}.
  The $S$ relation---that spheres evaluate to $0$---follows from the
  fact that $\eta\circ\iota=0$. The $T$ relation---that tori
  evaluate to $2$---corresponds to the composition
  \[
  1\stackrel{\iota}{\longrightarrow}x_+\stackrel{\Delta}{\longrightarrow}x_+\otimes
  x_-+x_-\otimes x_+-x_+\otimes x_+
  \stackrel{m}{\longrightarrow} x_- +x_--x_+\stackrel{\eta}{\longrightarrow}2.
  \]
  For the $4\mathit{Tu}$ relation, with notation as
  in~\cite[Proposition 7.2]{Bar-kh-tangle-cob}, it suffices to show
  that $L=R$. Computing,
  \begin{align*}
    L&=(\Delta\circ \iota)\otimes \iota\otimes\iota +
    \iota\otimes \iota\otimes (\Delta\circ\iota)\\
    &=\bigl[(x_+\otimes x_-+x_-\otimes x_+-x_+\otimes x_+)\otimes x_+\otimes
    x_+\bigr]\\
    &\qquad\qquad
    +\bigl[x_+\otimes x_+\otimes (x_+\otimes x_-+x_-\otimes
    x_+-x_+\otimes x_+)]\\
    &=x_+\otimes x_-\otimes x_+\otimes x_+ + x_-\otimes x_+\otimes
    x_+\otimes x_+ -x_+\otimes x_+\otimes x_+\otimes x_+\\
    &\qquad\qquad
    x_+\otimes x_+\otimes x_+\otimes x_- + x_+\otimes x_+\otimes
    x_-\otimes x_+ - x_+\otimes x_+\otimes x_+\otimes x_+\\
    &=x_{-+++}+x_{+-++}+x_{++-+}+x_{+++-}-2x_{++++}.
  \end{align*}
  Given a vector space $V$, let $s_{i,j}\co V^{\otimes n}\to
  V^{\otimes n}$ be the map which exchanges the $i\th$ and $j\th$
  factors. Then $R=s_{23}\circ L$; but $L$ is invariant
  under $s_{23}$.
\end{proof}

The following is well-known to experts:
\begin{prop}\label{prop:S-quasi-iso}
  Let $S$ be a connected cobordism from a knot $K_1$ to a knot
  $K_2$. Then the induced map $F_S\co \BNcx(K_1)\to \BNcx(K_2)$ is a
  quasi-isomorphism.
\end{prop}
\begin{proof}
  The analogous statement using the Lee deformation was proved by
  Rasmussen~\cite{Ras-kh-slice}.  Using Turner's change of basis
  (\Equation{change-of-basis}), Rasmussen's argument applies without
  essential changes.
\end{proof}

\subsection{The \texorpdfstring{$s$}{s} invariants}\label{subsec:BN-s-invt}
\begin{citethm}\label{citethm:MTV}
  \cite{MTV-Kh-s-invts} Let $\Field$ be a field of characteristic
  different from $2$, so Rasmussen's $s$ over $\Field$ is
  well-defined. Then the numbers
  \begin{align*}
    s_{\min}^\Field(K)&=\max\set{q\in2\Z+1}{i_*\co H_*(\Filt_q\BNcx)\to H_*(\BNcx)\cong \Field^2\text{
        surjective}} \\
    s_{\max}^\Field(K)&= \max\set{q\in 2\Z+1}{i_*\co H_*(\Filt_q\BNcx)\to H_*(\BNcx)\cong \Field^2\text{ nonzero}}
  \end{align*}
  defined using Bar-Natan's complex agree with Rasmussen's $s$
  invariants defined in~\cite[Definition 3.1]{Ras-kh-slice} (but using
  the field $\Field$ instead of $\QQ$).
\end{citethm}
\begin{proof}
  This is immediate from~\cite[Proposition 3.1]{MTV-Kh-s-invts}.
\end{proof}

\begin{prop}\label{prop:differ-by-two} Let $K$ be a knot and $\Field$ a field. Then
  $s_{\max}^\Field(K)=s_{\min}^\Field(K)+2$.
\end{prop}
\begin{proof}
  In the case that the $\characteristic(\Field)\neq 2$ this follows
  from \Citethm{MTV} and~\cite[Proposition
  3.3]{Ras-kh-slice} (but with $\Field$ used in place of $\QQ$). For
  the case that $\characteristic(\Field)=2$ we need a further
  argument.

  First, we
  show that $s_{\max}^\Field(K)\neq
  s_{\min}^\Field(K)$. There is an involution $I$ of the modulo $2$
  Bar-Natan Frobenius algebra defined by
  \begin{align*}
    I(\xbar)&=x_- & I(x_-)&=\xbar\\
    \shortintertext{or equivalently}
    I(x_+)&=x_+ & I(x_-)&=x_-+x_+.
  \end{align*}
  The map $I$ induces an involution $I_*\co \BNcx(K)\to
  \BNcx(K)$. Since $I$ exchanges $x_-$ and $\xbar$, $I_*$ is the
  nontrivial involution of $H_*(\BNcx(K))=\Field\oplus\Field$.

  Let $a\in \Filt_{s_{\min}^\Field(K)}\BNcx$ be a cycle so that
  $H_*(\BNcx(K))=\Field\langle a, I_*(a)\rangle$. In particular, $a$
  is chosen so that $a+I_*(a)$ represents a nontrivial homology class.
  The map $I$ respects the $q$-filtration; moreover, $I$ induces the
  identity map on the associated graded complex. It follows that the
  lowest-order terms of $a$ and $I(a)$ are the same. Thus, $a+I(a)\in
  \Filt_{q+2}\BNcx(K)$, so $s_{\max}^\Field(K)\geq s_{\min}^\Field(K)+2$.

  Next, we argue as in the proof of~\cite[Proposition
  3.3]{Ras-kh-slice} that $s_{\max}^\Field(K)\leq
  s_{\min}^{\Field}(K)+2$.  Let $U$ denote the unknot. Trivially,
  $H_*(\BNcx(U))=\BNcx(U)=\Field\langle \xbar,x_-\rangle$, with both
  $\xbar$ and $x_-$ lying in filtration level $q=-1$. Now, $\BNcx(U)$
  is a filtered ring, and choosing a basepoint $p$ on $K$ converts
  $\BNcx(K)$ and $H_*(\BNcx(K))$ into filtered modules over
  $\BNcx(U)$, where multiplication is filtered of degree $-1$. The
  element $a$, above, is homologous to one of the canonical generators
  $a'$; without loss of generality, suppose that in the generator
  $a'$, the component containing $p$ is labeled by $x_-$. Then
  $x_-I(a)=0$ and $x_-a=a$. Hence \[
  s_{\min}^\Field(K)=s(a)=s(x_-(a+I_*(a)))\geq
  s(x_-)+s(a+I_*(a))-1=s_{\max}^\Field(K)-2, \] as
  desired.
\end{proof}

In view of \Proposition{differ-by-two} write
$s^\Field(K)=s^\Field_{\min}(K)+1=s^\Field_{\max}(K)-1$.

\begin{corollary}\label{cor:s-bounds-genus}
  Let $K_1$ and $K_2$ be knots and let $S$ be a connected cobordism in
  $[0,1]\times S^3$ from $K_1$ to $K_2$. Then the invariant $s^\Field$
  defined using the Bar-Natan complex satisfies
  \[
  |s^\Field(K_1)-s^\Field(K_2)|\leq -\chi(S)
  \]
  In particular, for a knot $K$,
  \[
  |s^\Field(K)|\leq 2g_4(K).
  \]
\end{corollary}
\begin{proof}
  In the case that $\characteristic(\Field)\neq 2$, this follows from
  \Citethm{MTV} (i.e.,~\cite[Proposition 3.1]{MTV-Kh-s-invts})
  and Rasmussen's result~\cite[Theorem 1]{Ras-kh-slice} (except with
  $\Field$ in place of $\QQ$). But, using this as an opportunity to
  review Rasmussen's argument, which will be adapted to a slightly
  more complicated setting below, we give a direct proof.

  Let $q=s_{\min}^\Field(K_1)$. By \Citethm{BN-filt-cob},
  $F_S\from\BNcx(K_1)\to\BNcx(K_2)$ is a filtered map of filtration
  $\chi(S)$, so we have the following commutative diagrams:
  \[
  \vcenter{\hbox{
      \xymatrix{
        \Filt_q\BNcx(K_1)\ar[r]^-{F_S}\ar@{^(->}[d]_i & \Filt_{q+\chi(S)}\BNcx(K_2)\ar@{^(->}[d]^i\\ 
        \BNcx(K_1)\ar[r]_-{F_S} & \BNcx(K_2)
      }}}\text{ and }
  \vcenter{\hbox{
      \xymatrix{
        H_0(\Filt_q\BNcx(K_1))\ar[r]^-{F_S}\ar[d]_{i_*} & H_0(\Filt_{q+\chi(S)}\BNcx(K_2))\ar[d]^{i_*}\\ 
        H_0(\BNcx(K_1))\ar[r]^{\cong}_-{F_S} & H_0(\BNcx(K_2)).
      }}}
  \]
  Choose $a,b\in H_0(\Filt_q\BNcx(K_1))$ so that $i_*(a),i_*(b)\in
  H_0(\BNcx(K_1))$ form a basis. By \Proposition{S-quasi-iso},
  $F_S(i_*(a))=i_*(F_S(a))$ and
  $F_S(i_*(b))=i_*(F_S(b))$ in $H_0(\BNcx(K_2))$ also form a
  basis. So, $s^{\Field}_{\min}(K_2)\geq q+\chi(S)=s^{\Field}_{\min}(K_1)+\chi(S)$, or
  equivalently $-\chi(S)\geq s^{\Field}_{\min}(K_1)-s^{\Field}_{\min}(K_2)$. Viewing $S$
  as a cobordism from $K_2$ to $K_1$ instead and applying the same
  argument gives $-\chi(S)\geq s^{\Field}_{\min}(K_2)-s^{\Field}_{\min}(K_1)$. Thus,
  \[
  -\chi(S)\geq |s^{\Field}_{\min}(K_2)-s^{\Field}_{\min}(K_1)|.
  \]
  Now, applying \Proposition{differ-by-two} gives the first half of
  the result. The second half follows from the trivial computation
  that $s^\Field(U)=0$.
\end{proof}

\begin{corollary}\label{cor:concord-homo}
  The invariant $s^\Field(K)/2$ defines a homomorphism from the smooth
  concordance group to $\ZZ$.
\end{corollary}
\begin{proof}
  In the case that $\characteristic(\Field)\neq 2$ this follows
  from \Citethm{MTV} (i.e.,~\cite[Theorem 4.2]{MTV-Kh-s-invts})
  and Rasmussen's result~\cite[Theorem 2]{Ras-kh-slice}. For the general
  case, it follows from \Corollary{s-bounds-genus} that $s^\Field$ descends
  to a function on the smooth concordance group. So, it only remains
  to prove that this function is a homomorphism, i.e., that
  \[
  s^\Field(K_1\# K_2)=s^\Field(K_1)+s^\Field(K_2).
  \]
  This follows by the same argument as~\cite[Proposition
  3.12]{Ras-kh-slice}. The main points are that there is a short exact
  sequence
  \[
  0\to H_*(\BNcx(K_1\# K_2))\to H_*(\BNcx(K_1))\otimes
  H_*(\BNcx(K_2))\to H_*(\BNcx(K_1\#r(K_2)))\to 0
  \]
  where $r$ denotes the reverse and both maps in the sequence are
  filtered of $q$-degree $-1$ (\cite[Lemma 3.8]{Ras-kh-slice}); and
  that for the mirror $\overline{K}$ of a knot $K$,
  \[
  s_{\min}^\Field(\overline{K})=-s_{\max}^\Field(K)
  \]
  (\cite[Proposition 3.10]{Ras-kh-slice}). The proofs of both
  statements carry over to the Bar-Natan complex without change. The
  first statement implies that $s_{\min}^\Field (K_1\# K_2)\leq
  s_{\min}^\Field (K_1)+s_{\min}^\Field (K_2)+1$. The second then
  implies (by considering the mirrors) that
  $s_{\max}^\Field(K_1)+s_{\max}^\Field(K_2) \leq
  s_{\max}^\Field(K_1\#K_2)+1$.  Using \Proposition{differ-by-two}
  then gives $s_{\min}^\Field (K_1)+s_{\min}^\Field (K_2)+4\leq
  s_{\min}^\Field (K_1\#K_2)+3$. Combining this with the first
  inequality gives $s_{\min}^\Field (K_1\# K_2)=s_{\min}^\Field
  (K_1)+s_{\min}^\Field (K_2)+1=s^\Field(K_1)+s^\Field(K_2)-1$ and
  $s_{\max}^\Field(K_1\#
  K_2)=s_{\max}^\Field(K_1)+s_{\max}^\Field(K_2)-1=s^\Field(K_1)+s^\Field(K_2)+1$,
  proving the result.
\end{proof}

\section{Cobordism maps commute with cohomology
  operations}\label{sec:cob-commutes-st}
We start this section with a brief introduction to the Khovanov homotopy
type in \Subsection{review}.
The rest of this section is devoted to showing that the cobordism maps on Khovanov
homology commute with stable cohomology operations (like Steenrod
squares). To prove this, we simply need to associate a map of Khovanov
spectra to each cobordism and verify that the induced map on
cohomology agrees with the homomorphism given
in~\cite{Kho-kh-tangles,Jac-kh-cobordisms}.

Like the cobordism maps on Khovanov homology, the cobordism maps on
the Khovanov spectra are defined by composing maps associated to
elementary cobordisms. We conjecture that up to (stable) homotopy the
map of Khovanov spectra is independent of the decomposition into
elementary cobordisms---i.e., is an isotopy invariant of the
cobordism---but we will not show that here; see also \Remark{sign-ambiguity}.

In the process of proving that the Khovanov homotopy type is a knot
invariant, in~\cite{RS-khovanov} we associated maps to Reidemeister
moves. These maps were homotopy equivalences, and in particular induce
isomorphisms on Khovanov homology; but we did not verify
in~\cite{RS-khovanov} that these isomorphisms agree with the standard
ones. So, we give this verification in \Subsection{Reidemeister}.

Given this, it remains to define maps associated to cups, caps, and
saddles. Doing so is fairly straightforward;
see \Subsection{cup-cap-sad}. With these ingredients in hand,
\Theorem{coho-op-commute} follows readily;
see \Subsection{prove-coho-op-commute}.

\subsection{A brief review of the Khovanov homotopy type}\label{subsec:review}
In this subsection, we summarize the construction of the
Khovanov suspension spectrum from~\cite{RS-khovanov}. 
The construction is somewhat involved, and at places, fairly
technical, so we only present the general outline, and highlight some
of the salient features. The reader may find it helpful to consult the
examples in~\cite[Section~\ref*{KhSp:sec:examples}]{RS-khovanov} in
conjunction with the discussion here.

Fix a link $L$; the construction of the suspension spectrum
$\KhSpace(L)$ depends on several choices that are listed
in \Subsection{Reidemeister}. Of them, we treat
choice~(\ref{item:choice:ladybug}), namely the choice of a ladybug
matching, as a global
choice. Choices~(\ref{item:choice:diagram})--(\ref{item:choice:sign})
essentially correspond to the choice of a link diagram $D$; those are the
usual choices that one makes in defining the Khovanov chain complex;
after making those choices, we can talk about a Khovanov chain complex
$\KhCx(D)$, which is a chain complex equipped with distinguished set
of generators, which we often refer to as the Khovanov generators.

The suspension spectrum $\KhSpace(L)$ is constructed as (the formal
desuspension of) the suspension spectrum of a CW complex
$\Realize{\KhFlowCat(D)}$. The cells in $\Realize{\KhFlowCat(D)}$
(except the basepoint) canonically
correspond to the Khovanov generators; furthermore, the
correspondence induces an isomorphism between the reduced cellular
cochain complex of $\Realize{\KhFlowCat(D)}$ and $\KhCx(D)$ (after an
overall grading shift).

The central idea is to construct an intermediate object
$\KhFlowCat(D)$, which is a framed flow category. We define a partial
order on the Khovanov generators by declaring $\bm{b}\prec\bm{a}$ if
there is a sequence of differentials in $\KhCx(D)$ from $\bm{b}$ to
$\bm{a}$. To such a pair $\bm{b}\prec\bm{a}$, we associate a
framed $(\gr_h(\bm{a})-\gr_h(\bm{b})-1)$-dimensional moduli space
$\Moduli(\bm{a},\bm{b})$ subject to the following conditions:
\begin{enumerate}
\item If $\bm{a}$ appears in the Khovanov differential applied to
  $\bm{b}$ with coefficient $n_{\bm{ab}}$, then $\Moduli(\bm{a},\bm{b})$
  consists of $n_{\bm{ab}}$ points, counted with sign. (Recall that a
  framed $0$-manifold is a disjoint union of signed points.)
\item\label{item:coherence} If $\bm{c}\prec\bm{b}\prec\bm{a}$, then
  $\Moduli(\bm{a},\bm{b})\times\Moduli(\bm{b},\bm{c})$ is identified
  with a certain subset of $\bdy\Moduli(\bm{a},\bm{c})$, and the
  framings are coherent. (See,
  e.g.,~\cite[Definition~\ref*{KhSp:def:flow-cat}]{RS-khovanov} for a
  precise version of this condition.)
\end{enumerate}

To such a framed flow category $\KhFlowCat(D)$, one can associate an
explicit CW complex $\Realize{\KhFlowCat(D)}$. The construction was
introduced by Cohen-Jones-Segal in
\cite[pp.~309--312]{CJS-gauge-floerhomotopy}, and is
described in more detail in
\cite[Subsection~\ref*{KhSp:sec:flow-to-space}]{RS-khovanov}. If one
has a framed $k$-dimensional manifold in $\R^n$, by the
Pontryagin-Thom construction, one gets a (stable) map from $S^n$ to
$S^{n-k}$, and thereby a CW complex with a $0$-cell, an $(n-k)$-cell,
and an $(n+1)$-cell. The Cohen-Jones-Segal construction is a refined
version of the Pontryagin-Thom 
construction, allowing one to describe arbitrary CW complexes at the
cost of working with manifolds with corners (organized into flow
categories). One needs to choose a few parameters in order to pass
from a framed flow category to a CW complex; these parameters are are listed as
choice~(\ref{item:choice:ABepR}) in \Subsection{Reidemeister}.

So all that remains is to construct the Khovanov flow category
$\KhFlowCat(D)$. The Khovanov chain complex $\KhCx(D)$ is modelled
after the cube chain complex $\CubeCx(N)$, which is obtained from the
cube $(\Z\to\Z)^{\otimes N}$ after infusing the arrows with a sign
assignment that makes every face anti-commute. One can
define a framed flow category for the cube by defining the
moduli spaces to be permutahedra. This is
choice~(\ref{item:choice:neat}) from \Subsection{Reidemeister}. We
define the Khovanov flow category by modelling it on the cube flow
category $\CubeFlowCat(N)$, via
choice~(\ref{item:choice:perturbation}); in particular, we ensure that
the each moduli space is a disjoint union of permutahedra.

More concretely, if $\bm{b}\prec\bm{a}$, and
$\gr_h(\bm{a})-\gr_h(\bm{b})=1$, we define $\Moduli(\bm{a},\bm{b})$ to
be a point (with the framing determined by the sign assignment). If
$\bm{b}\prec\bm{a}$, and $\gr_h(\bm{a})-\gr_h(\bm{b})=2$, there could
be two or four broken flowlines from $\bm{b}$ to $\bm{a}$. In the
former case, we define $\Moduli(\bm{a},\bm{b})$ to be an interval. In
the latter case, we define it to be a disjoint union of two intervals;
there is a choice of matching that one needs to do in this case, and
that is precisely the choice of the ladybug matching. The rest of the
construction proceeds inductively: When we construct the
$n$-dimensional moduli space $\Moduli(\bm{a},\bm{b})$, all the lower
dimensional moduli spaces in the Khovanov flow category
$\KhFlowCat(D)$ have already been constructed, and the lower
dimensional moduli spaces admit covering maps to the corresponding
moduli spaces in the cube flow category $\CubeFlowCat(N)$. Therefore,
by (the precise version of) Condition~(\ref{item:coherence}) above, $\bdy\Moduli(\bm{a},\bm{b})$ has already been constructed,
and it admits a covering map to the boundary of the $n$-dimensional
permutahedron. When $n\geq 3$, by simple-connectedness, this forces
$\bdy\Moduli(\bm{a},\bm{b})$ to be a disjoint union of boundaries of
permutahedra, and therefore, we can define $\Moduli(\bm{a},\bm{b})$ to
be the corresponding disjoint union of permutahedra. When $n=2$, there
is a possible obstruction; however, an exhaustive case check
annihilates that possibility. Thus we can define the Khovanov flow
category, and via the Cohen-Jones-Segal construction, the Khovanov
suspension spectrum $\KhSpace(L)$.

\subsection{The Reidemeister maps agree}\label{subsec:Reidemeister}
Fix a link $L$. Recall from~\cite[Definition~\ref*{KhSp:def:Kh-space} and Section~\ref*{KhSp:sec:invariance}]{RS-khovanov} that the construction
of the suspension spectrum
$\KhSpace(L)$ depends on several choices:
\begin{enumerate}
\item\label{item:choice:ladybug} A choice of ladybug matching (left or right).
\item\label{item:choice:diagram} An oriented link diagram $D$ for $L$,
  with $N$ crossings.
\item\label{item:choice:crossings} An ordering of the crossings of $D$.
\item\label{item:choice:sign} A sign assignment $s$ for the cube $\Cube(N)$.
\item \label{item:choice:neat}A neat embedding $\iota$ and a framing $\Frame$ for the cube
  flow category $\CubeFlowCat(N)$ relative to $s$.
\item\label{item:choice:perturbation} A framed neat embedding $\kappa$ of the Khovanov flow category
  $\KhFlowCat(L)$ relative to some $\TupV{d}$. This framed neat
  embedding is a perturbation of $(\iota,\Frame)$.
\item\label{item:choice:ABepR} Integers $A,B$ and real
  numbers $\ep,R$.
\end{enumerate}
It is proved in~\cite[Section~\ref*{KhSp:sec:invariance}]{RS-khovanov}
that, up to homotopy equivalence, $\KhSpace(L)$ is independent
of these auxiliary choices. We will view the ladybug matching as a
global choice.  The goal of this section is to prove that, on the
level of homology, the rest of these homotopy equivalences agree with
the isomorphisms on Khovanov homology.

Recall that $\KhSpace(L)$ is a formal de-suspension of the realization
$\Realize{\KhFlowCat(L)}$ in the sense
of~\cite[Definition~\ref*{KhSp:def:flow-gives-space}]{RS-khovanov} of
the Khovanov flow category $\KhFlowCat(L)$.  As noted, this
realization depends on the auxiliary choices above. For $D$ a link
diagram, let
${\Realize{\KhFlowCat(D)}}_{o,s,\iota,\Frame,\kappa,A,B,\ep,R}$ denote
the Khovanov space defined using the ordering $o$ of the crossings of
$D$, sign assignment $s$, framed embedding $(\iota,\Frame)$ of the
cube flow category, perturbation $\kappa$ of $\iota$, integers $A$ and
$B$ and real numbers $\epsilon$ and $R$.

For any set of auxiliary choices, there is a canonical identification
of the cells of $\Realize{\KhFlowCat(D)}$ and generators of the
Khovanov complex $\KhCx(D)$, intertwining the cellular cochain
differential on $\Realize{\KhFlowCat(D)}$ and the Khovanov
differential on $\KhCx(D)$. This gives a canonical identification
between
$\wt{H}^*({\Realize{\KhFlowCat(D)}}_{o,s,\iota,\Frame,\kappa,A,B,\ep,R})$
and $\Kh(D)$. We need to show that the stable homotopy equivalences
associated to changes of auxiliary data respect this
identification. We start with
choices~(\ref{item:choice:neat})--(\ref{item:choice:ABepR}):
\begin{prop}\label{prop:extra-choices}
  Let $D$ be a link diagram. Then for any two choices of auxiliary
  data $(\iota,\Frame,\kappa,A,B,\ep,R)$ and
  $(\iota',\Frame',\kappa',A',B',\ep',R')$ the stable homotopy
  equivalence
  \[
  {\Realize{\KhFlowCat(D)}}_{o,s,\iota,\Frame,\kappa,A,B,\ep,R}\simeq {\Realize{\KhFlowCat(D)}}_{o,s,\iota',\Frame',\kappa',A',B',\ep',R'}
  \]
  furnished
  by~\cite[Proposition~\ref*{KhSp:prop:choice-independent}]{RS-khovanov}
  induces the identity map on Khovanov homology.
\end{prop}
\begin{proof}
  It is immediate from the proof
  of~\cite[Lemma~\ref*{KhSp:lem:CW-indep-ep-R-framing}]{RS-khovanov}
  that the maps associated to changing $\epsilon$ and $R$ are
  isomorphisms of CW complexes respecting the identification of cells
  with generators of $\KhCx$; in particular, they induce the identity
  map on Khovanov homology.  As in the proof
  of~\cite[Lemma~\ref*{KhSp:lem:CW-indep-A-B}]{RS-khovanov},
  increasing $A'$, decreasing $B'$ or increasing $\TupV{d}$ has the
  effect of suspending the CW complex, and the map of realizations is
  the identity map
  \[
  \Sigma^{d}
  {\Realize{\KhFlowCat(D)}}_{o,s,\iota,\Frame,\kappa,A,B,\ep,R}\stackrel{\cong}{\longrightarrow} {\Realize{\KhFlowCat(D)}}_{o,s,\iota,\Frame,\kappa,A',B',\ep,R}.
  \]
  In particular, the induced map on Khovanov homology is the identity.

  By \cite[Lemmas~\ref*{KhSp:lem:perturb-connect-framings}
  and~\ref*{KhSp:lem:change-framing-cube}]{RS-khovanov}, any choices
  of $\kappa$, $\iota$ and $\Frame$ lead to isotopic framings of
  $\KhFlowCat(D)$. Again, from the proof
  of~\cite[Lemma~\ref*{KhSp:lem:CW-indep-ep-R-framing}]{RS-khovanov},
  the map associated to isotoping the framing of $\KhFlowCat(D)$
  is an isomorphism of CW complexes respecting the identification of
  cells with generators of $\KhCx$. So, again, these maps induce the
  identity map on Khovanov homology.
\end{proof}

The fact that the isomorphisms agree for changes in sign assignment
and ordering of crossings is a bit more subtle. First, recall that a
\emph{sign assignment} is a choice of signs for the edges of the
\emph{cube chain complex}
\[
\CubeCx(N)=(\ZZ\stackrel{\cong}{\longrightarrow}\ZZ)^{\otimes N}
\]
so that $\bdy^2=0$, i.e., so that each face anti-commutes.
In~\cite{RS-khovanov}, the sign assignment and ordering of crossings
enter the construction of $\Realize{\KhFlowCat(D)}$ as follows. The
ordering of crossings specifies a map from Khovanov generators to
vertices of the hypercube $\{0,1\}^N$. Using this map, the sign
assignment then specifies framings for the $0$-dimensional moduli
spaces in $\KhFlowCat(D)$. Thus, changing the ordering of the
crossings is equivalent to changing the sign assignment, and we
prefer to view the operation in the latter way (cf.~\cite[Proof of
Proposition~\ref*{KhSp:prop:choice-independent}]{RS-khovanov}).

In the Khovanov homology literature, it seems to be more typical to
phrase results in terms of the standard sign assignment, but with
different orderings of crossings; see, for instance,~\cite[p.\
1457]{Bar-kh-tangle-cob}. In our style, the change of ordering
homomorphism of~\cite[p.\ 1457]{Bar-kh-tangle-cob} is given as
follows. Given sign assignments $\signass$, $\signass'$ choose a map
$\tsignass$ from the vertices of $\CubeCx(N)$ to $\{\pm 1\}$ so that
the map $(\CubeCx(N),\bdy_{\signass})\to
(\CubeCx(N),\bdy_{\signass'})$ given by $v\mapsto \tsignass(v)\cdot v$
is a chain map. (We will call $\tsignass$ a \emph{gauge transformation}
from $\signass$ to $\signass'$.)
Then, the map of Khovanov complexes takes a generator $x$ lying over a
vertex $v$ of the hypercube to $\tsignass(v)\cdot x$.  The following
lemma is relevant:
\begin{lemma}\label{lem:sign-map-nearly-unique}
  Given sign assignments $\signass$, $\signass'$ for $\CubeCx(N)$
  there are exactly two gauge transformations $\tsignass_1,\tsignass_2$
  from $\signass$ to $\signass'$. Moreover,
  $\tsignass_2=-\tsignass_1$.
\end{lemma}
\begin{proof}
  This is a straightforward induction argument, showing that the value
  of $\tsignass$ on the vertex $(0,\dots,0)$, say, and the fact that
  $v\mapsto \tsignass(v)\cdot v$ is a chain map uniquely determine
  $\tsignass$.
\end{proof}

\begin{prop}\label{prop:signs}
  Let $\Phi_{\signass_0,\signass_1}$ denote the map of Khovanov
  homotopy types associated
  in~\cite[Proposition~\ref*{KhSp:prop:choice-independent}]{RS-khovanov}
  to a change of sign assignments  
  and $F_{\signass_0,\signass_1}$ the map of Khovanov chain complexes
  induced by a change of sign assignments, as alluded to in~\cite[p.\ 
  1457]{Bar-kh-tangle-cob}. Then the following diagram commutes up to
  sign:
  \[
  \xymatrix{
    \wt{H}^i(\KhSpace^j(D,\signass_0))\ar[d]_\cong\ar[r]^{\Phi^*_{\signass_0,\signass_1}}_\cong &
    \wt{H}^i(\KhSpace^j(D,\signass_1))\ar[d]^\cong\\
    \Kh^{i,j}(D,\signass_0)\ar[r]^{F_{\signass_0,\signass_1}}_\cong & \Kh^{i,j}(D,\signass_1).
  }
  \]
\end{prop}
\begin{proof}
  By \Lemma{sign-map-nearly-unique}, it suffices to show that the map
  $\Phi_{\signass_0,\signass_1}^*$ on cochain complexes takes each
  generator to $\pm$ itself, where the signs are determined by a map
  of $\CubeCx(N)$. We recall the definition of
  $\Phi_{\signass_0,\signass_1}$. Consider the diagram $D\amalg U$
  obtained by taking the disjoint union of $D$ with a $1$-crossing diagram
  for the unknot. Choose $U$ so that the $0$-resolution of $U$ has two
  components and the $1$-resolution of $U$ has one component. Let
  $\KhFlowCat(D\amalg U)_+$ denote the full subcategory generated by
  the objects in which the (one or two) circles corresponding to $U$
  are labeled by $x_+$. Write $\KhFlowCat(D\amalg U)_0$ to be the
  subcategory of $\KhFlowCat(D\amalg U)_+$ in which we take the
  $0$-resolution of $U$ and $\KhFlowCat(D\amalg U)_1$ to be the
  subcategory of $\KhFlowCat(D\amalg U)_+$ in which we take the
  $1$-resolution of $U$. Then each of $\KhFlowCat(D\amalg U)_0$ and
  $\KhFlowCat(D\amalg U)_1$ is isomorphic to $\KhFlowCat(D)$.  Choose
  a sign assignment $\signass$ for $D\amalg U$ so that the induced sign
  assignment for $\KhFlowCat(D\amalg U)_i$ is $\signass_i$. Thus, if
  we choose the embedding data compatibly,
  $\Realize{\KhFlowCat(D\amalg
    U)_0}={\Realize{\KhFlowCat(D)}}_{\signass_0}$ and $\Realize{\KhFlowCat(D\amalg
    U)_1}=\Sigma{\Realize{\KhFlowCat(D)}}_{\signass_1}$.  Moreover:
  \begin{enumerate}
  \item\label{item:sign-cofib} There is a cofibration sequence
    \[
    {\Realize{\KhFlowCat(D)}}_{\signass_0}\to
    {\Realize{\KhFlowCat(D\amalg U)_+}}_{\signass}\to
    \Sigma{\Realize{\KhFlowCat(D)}}_{\signass_1}.
    \]
  \item\label{item:sign-contr} The cohomology $\wt{H}^*({\Realize{\KhFlowCat(D\amalg
      U)_+}}_{\signass})$ is trivial, so ${\Realize{\KhFlowCat(D\amalg
      U)_+}}_{\signass}$ is contractible.
  \end{enumerate}
  The map $\Phi_{\signass_0,\signass_1}\co \KhSpace^j(D,\signass_1)\to
  \KhSpace^j(D,\signass_0)$ is the Puppe map associated
  to the cofibration sequence in~(\ref{item:sign-cofib}), which is an
  isomorphism by~(\ref{item:sign-contr}). (The suspension coming from
  the Puppe construction cancels with the suspension in
  $\Realize{\KhFlowCat(D\amalg
    U)_1}=\Sigma{\Realize{\KhFlowCat(D)}}_{\signass_1}$.)

  The induced map $\Phi^*_{\signass_0,\signass_1}$ is the connecting
  homomorphism
  \[
  \partial\co
  \Kh(D,\signass_0)=\wt{H}^*({\Realize{\KhFlowCat(D)}}_{\signass_0})\to
  \wt{H}^*({\Realize{\KhFlowCat(D)}}_{\signass_1})=\Kh(D,\signass_1)
  \]
  associated to the short exact sequence of chain complexes
  \[
  0 \to \KhCx(D,\signass_1)\stackrel{i}{\longrightarrow} \KhCx(D\amalg
  U,s)_+ \stackrel{p}{\longrightarrow} \KhCx(D,\signass_0)\to 0.
  \]
  For $v\in\{0,1\}^N$ let $\tsignass(v)$ be the sign assigned by
  $\signass$ to the edge from $(v,0)$ to $(v,1)$. Explicitly, the
  relevant maps are given by
  \begin{align*}
    &i((v,x))=(v,x)\otimes x_+\qquad p((v,x)\otimes x_+\otimes x_+)=(v,x)  \qquad p((v,x)\otimes x_+)=0\\
    &\diff((v,x)\otimes x_+\otimes x_+)=\tsignass(v)(v,x)\otimes
    x_++\diff_{\signass_0}(v,x)\otimes x_+\otimes x_+\\
    &\diff((v,x)\otimes
    x_+)=\diff_{\signass_1}(v,x)\otimes x_+.
  \end{align*}
  (Here, the $x_+$'s are labels for the $0$- or $1$-resolution of $U$,
  and $\diff$ is the Khovanov differential for $\KhCx(D\amalg
  U,s)_+$,
  and $(v,x)$ is a Khovanov generator for $\KhCx(D)$ lying over the vertex $v\in
  \{0,1\}^n$.) So, the connecting homomorphism $\partial$ is given by
  $\partial(v,x)=\tsignass(v)(v,x)$, which is as alluded to
  in~\cite[p.\ 1457]{Bar-kh-tangle-cob}.
\end{proof}

In light of Propositions~\ref{prop:extra-choices}
and~\ref{prop:signs}, it is safe to drop the auxiliary data
$(o,s,\iota,\Frame,\kappa,A,\allowbreak B,\ep,R)$ from both the notation and the
discussion, and we will do so for the rest of the paper.

\begin{prop}\label{prop:Reid-agree}
  Let $D$ and $D'$ be link diagrams representing $L$, and choose a
  sequence of Reidemeister moves connecting $D$ and $D'$. Let $\Phi\co
  \KhSpace^j(D')\to \KhSpace^j(D)$ denote the stable homotopy equivalence
  given by~\cite[Theorem~\ref*{KhSp:thm:kh-space}]{RS-khovanov} and let
  $F\co \Kh^{i,j}(D)\to \Kh^{i,j}(D')$ denote the isomorphism given
  by~\cite{Kho-kh-tangles}, \cite{Jac-kh-cobordisms}
  or~\cite{Bar-kh-tangle-cob}. Then the following diagram commutes up
  to sign:
  \[
  \xymatrix{
    \wt{H}^i(\KhSpace^j(D))\ar[r]^{\Phi^*}_\cong\ar[d]_\cong &
    \wt{H}^i(\KhSpace^j(D'))\ar[d]^\cong\\
    \Kh^{i,j}(D)\ar[r]^F_\cong & \Kh^{i,j}(D').
  }
  \]
  (Here, the vertical isomorphisms are also given by~\cite[Theorem~\ref*{KhSp:thm:kh-space}]{RS-khovanov}.)
\end{prop}

\begin{figure}
  \centering
  \begin{overpic}[tics=5]{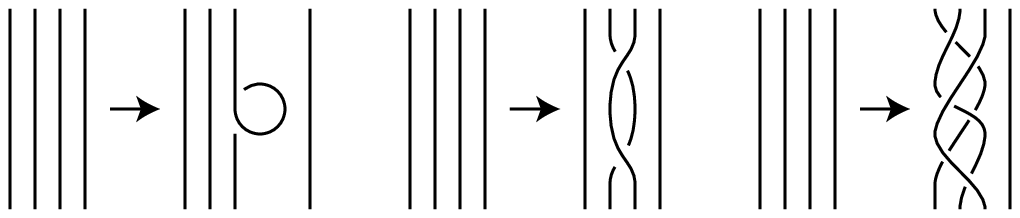}
    \put(5,3){$T_1$}
    \put(24,3){$T'_1$}
    \put(43,3){$T_2$}
    \put(59,3){$T'_2$}
    \put(76,3){$T_3$}
    \put(92,3){$T'_3$}
  \end{overpic}  
  \caption{\textbf{Tangles for Reidemeister moves.} The number of
    vertical lines and which of them are involved in the Reidemeister
    move are allowed to vary (but the number of lines must be even).}
  \label{fig:Reidemeister}
\end{figure}

In~\cite{Kho-kh-tangles}, Khovanov associated to any even integer $2m$
a ring $H^m$, and to any $(2m_1,2m_2)$-tangle $T$ a bimodule
$\mathcal{M}(T)$ over the rings $H^{m_1}$ and $H^{m_2}$. As he
observed in~\cite{Kho-kh-cob}, these tangle invariants can be used to
state and prove locality properties for the cobordism
maps. \Proposition{Reid-agree} will follow easily from the next lemma,
which is along the same lines:
\begin{lemma}\label{lem:maps-from-tangles}
  The maps on Khovanov homology associated
  in~\cite[Propositions~\ref*{KhSp:prop:RI}--\ref*{KhSp:prop:RIII}]{RS-khovanov}
  to Reidemeister moves are induced by maps of tangles in the
  following sense. Let $T_k$ and $T_k'$ be the tangles shown in
  \Figure{Reidemeister}, so $T_k$ differs from $T_k'$ by a
  Reidemeister move. Suppose that $D$ and $D'$ are knot diagrams so
  that $D_k'$ is obtained by replacing a copy of $T_k$ inside $D$ by a
  copy of $T_k'$. Let
  \[
  \Phi_k^*\co \Kh^{i,j}(D')\cong \wt{H}^{i}(\KhSpace^j(D'))\to
  \wt{H}^{i}(\KhSpace^j(D))\cong \Kh^{i,j}(D)
  \]
  be the map induced by the map of spaces $\Phi$ defined
  in~\cite[Proposition~\ref*{KhSp:prop:RI}--\ref*{KhSp:prop:RIII}]{RS-khovanov}. Then
  there are bimodule maps $F_k\co \mathcal{M}(D'_k)\to
  \mathcal{M}(D_k)$ so that
  \[
  \Phi_k^* = \Id\otimes F_k\co H_*(\mathcal{M}(D\setminus T_k)\otimes
  \mathcal{M}(T_k'))\to H_*(\mathcal{M}(D\setminus T_k)\otimes
  \mathcal{M}(T_k)).
  \]
\end{lemma}
\begin{proof}
  It is straightforward to verify that the maps (on homology) given
  in~\cite[Section 6]{RS-khovanov} make sense on the level of
  tangles. The key point is that the maps are defined locally, by
  canceling acyclic subcomplexes and quotient complexes given by
  requiring certain circles in $T_k'$ to be decorated by one of $x_+$
  or $x_-$. The same definitions define acyclic subcomplexes and
  quotient complexes of
  $\mathcal{M}(T_k')$.
\end{proof}

\begin{proof}[Proof of \Proposition{Reid-agree}]
  It suffices to check the result when $D$ and $D'$ differ by a single
  Reidemeister move. In this case, by \Lemma{maps-from-tangles}, the
  maps $\Phi^*$ are induced by isomorphisms of Khovanov's tangle
  invariants. By definition, the maps $F$ are also induced by
  isomorphisms of tangle invariants. But the tangles $T_k$, $T'_k$
  involved are invertible, so by~\cite[Corollary
  2]{Kho-kh-cob} any two such isomorphisms agree up to sign.
\end{proof}

\subsection{Maps associated to cups, caps and
  saddles}\label{subsec:cup-cap-sad}
Other than Reidemeister moves, there are three kinds of elementary
cobordisms of knots: cups, saddles and caps, which correspond to index
$0$, $1$ and $2$ critical points of Morse functions, respectively.
\subsubsection{Cups and caps}
Let $L$ be a link diagram and $L'=L\amalg U$ the disjoint union of $L$ and an
unknot; place $U$ so that it does not introduce new crossings. Passing
from $L$ to $L'$ corresponds to a particular elementary cobordism, a
\emph{cup}, while passing from $L'$ to $L$ corresponds to a different
elementary cobordism, a \emph{cap}. There are maps
\begin{align*}
  F_\cup&\co \Kh^{i,j}(L)\to \Kh^{i,j+1}(L')\\
  F_\cap&\co \Kh^{i,j}(L')\to \Kh^{i,j+1}(L)
\end{align*}
on Khovanov homology associated to a cup and a cap, respectively.  We
want to define maps
\begin{align*}
  \Phi_\cup&\co \KhSpace^{j+1}(L')\to \KhSpace^{j}(L)\\
  \Phi_\cap&\co \KhSpace^{j+1}(L)\to \KhSpace^{j}(L')
\end{align*}
associated to a cup and a cap so that the induced maps on cohomology
are $F_\cup$ and $F_\cap$.

In each resolution $L'_v$ of $L'$ there is a component $U_v$
corresponding to $U$. We can write $\KhCx(L')=\KhCx(L)_+\oplus
\KhCx(L)_-$, where $\KhCx(L)_+$ (\respectively $\KhCx(L)_-$) has basis
those generators of $\KhCx(L')$ in which $U_v$ is labeled by $x_+$
(\respectively $x_-$). Each of $\KhCx(L)_+$ and $\KhCx(L)_-$ are
canonically isomorphic to $\KhCx(L)$. The cobordism maps on Khovanov
homology associated to cups and caps are defined on the chain level by
\begin{align*}
  F_\cup&\co
  \KhCx(L)\stackrel{\cong}{\longrightarrow}\KhCx(L)_+\hookrightarrow
  \KhCx(L')\\
  F_\cap&\co \KhCx(L')\twoheadrightarrow \KhCx(L')_-
  \stackrel{\cong}{\longrightarrow} \KhCx(L),
\end{align*}
to be the inclusion and projection,
respectively~\cite[Figure 15]{Jac-kh-cobordisms}.
Note that in the special case that $L$ is empty these maps restrict to
the unit and counit on $H^*(S^2)$, respectively.

Similarly, the flow category $\KhFlowCat(L')$ is a disjoint union
$\KhFlowCat(L')=\KhFlowCat(L)_+\amalg \KhFlowCat(L)_-$, where
$\KhFlowCat(L)_+$ (\respectively $\KhFlowCat(L)_-$) is the full
subcategory of $\KhFlowCat(L')$ whose objects are decorated
resolutions $(v,x)$ with $U_v$ labeled by $x_+$ (\respectively
$x_-$). Thus, $\Realize{\KhFlowCat(L)}\cong
\Realize{\KhFlowCat(L)_+}\vee \Realize{\KhFlowCat(L)_-}$.  Each of
$\KhFlowCat(L)_+$ and $\KhFlowCat(L)_-$ are canonically isomorphic to
$\KhFlowCat(L)$. Define
\begin{align*}
  \Phi_\cup &\co
  \Realize{\KhFlowCat(L')}=\Realize{\KhFlowCat(L)_+}\vee
  \Realize{\KhFlowCat(L)_-}\twoheadrightarrow\Realize{\KhFlowCat(L)_+}\stackrel{\cong}{\longrightarrow}\Realize{\KhFlowCat(L)}\\
  \Phi_\cap & \co
  \Realize{\KhFlowCat(L)}\stackrel{\cong}{\longrightarrow}\Realize{\KhFlowCat(L)_-}\hookrightarrow \Realize{\KhFlowCat(L)_+}\vee
  \Realize{\KhFlowCat(L)_-}=\Realize{\KhFlowCat(L')}
\end{align*}
to be the projection and inclusion, respectively.

\begin{lemma}\label{lem:cup-cap-agree}
  The map on cohomology induced by $\Phi_\cup$ (\respectively $\Phi_\cap$) is
  the cobordism maps $F_\cup$ (\respectively $F_\cap$) associated
  in~\cite{Kho-kh-tangles,Jac-kh-cobordisms} to the cup (\respectively
  cap) cobordism.
\end{lemma}
\begin{proof}
  This is immediate from the definitions.
\end{proof}

\subsubsection{Saddles}
The remaining elementary cobordism is a \emph{saddle}, which
corresponds to making the local change shown
\Figure{saddle}. The map $F_s$ on Khovanov homology associated
to a saddle is defined as follows. Let $L_0$ and $L_1$ be the link
diagrams before and after the saddle. Let $L$ be the link diagram
obtained by replacing the region in which the saddle move is occurring
by a crossing $c$, as in \Figure{saddle}. Then $\KhCx(L_1)$ is a
subcomplex of $\KhCx(L)$ with corresponding quotient complex
$\KhCx(L_0)$. The map $F_s$ is defined to be the connecting
homomorphism in the long exact sequence
\[
\dots\to \KhCx^{i-1,j-1}(L_1)\to \KhCx^{i+a,j+b}(L)\to
\KhCx^{i,j}(L_0)\stackrel{F_s}{\longrightarrow} \KhCx^{i,j-1}(L_1)\to\dots,
\]
where $a$ and $b$ are integers which depend on how the orientation for
$L$ is chosen.
Equivalently, $F_s$ is the map occurring in the skein exact sequence
associated to the crossing $c$; see, for instance,~\cite[p.\
1472]{Bar-kh-tangle-cob}. (Here, the gradings on $\KhCx(L_0)$ and
$\KhCx(L_1)$ relate nicely because the saddle is oriented; and the
grading on $\KhCx(L)$ does not relate as nicely to these, because the
orientations of $L_0$ and $L_1$ do not agree with an orientation
of $L$.)

\begin{figure}
  \centering
  \begin{overpic}[tics=5, scale=.75]{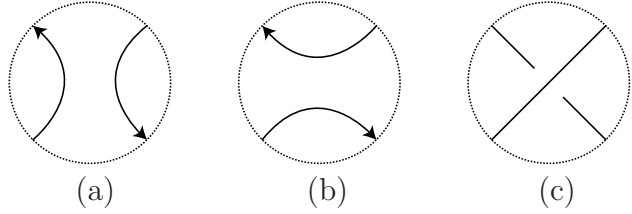}
    \put(13,0){(a)}
    \put(48,0){(b)}
    \put(83,0){(c)}
  \end{overpic}
  \caption{\textbf{A saddle.} Parts (a) and (b) are related by an
    oriented saddle move. (There is another valid oriented saddle,
    obtained by reflecting the pictures vertically.) Notice that (a) is
    the $0$-resolution of the crossing shown in (c), while (b) is the
    $1$-resolution of the crossing shown in (c). There is no way to
    orient the crossing in (c) coherently with (a) or (b).}
  \label{fig:saddle}
\end{figure}

Essentially the same construction carries through on the space
level. Briefly, a space-level version of the skein sequence is given
in~\cite[Section~\ref*{KhSp:sec:skein}]{RS-khovanov}, and we define
$\Phi_s$ to be the Puppe map associated to this sequence. In more detail,
the flow category $\KhFlowCat(L_0)$ is a downward-closed subcategory
of $\KhFlowCat(L)$ (in the sense
of~\cite[Definition~\ref*{KhSp:def:downward-closed}]{RS-khovanov}),
with corresponding upward-closed subcategory $\KhFlowCat(L_1)$. So,
there is a cofibration sequence
\[
\Realize{\KhFlowCat(L_0)}\to \Realize{\KhFlowCat(L)} \to \Realize{\KhFlowCat(L_1)}.
\]
The Puppe construction gives a map
\[
\Phi_s\co \Realize{\KhFlowCat(L_1)}\to \Sigma\Realize{\KhFlowCat(L_0)},
\]
and we define this to be the map of spaces associated to a saddle
cobordism.
Putting in the gradings, the map $\Phi_s$ has the form
\[
\Phi_s\co \KhSpace^{j}(L_1)\to \KhSpace^{j+1}(L_0).
\]

\begin{lemma}\label{lem:saddle-agree}
  The map on cohomology induced by $\Phi_s$ is the cobordism maps $F_s$
  associated in~\cite{Kho-kh-tangles,Jac-kh-cobordisms} to the saddle
  cobordism.
\end{lemma}
\begin{proof}
  Again, this is immediate from the definitions.
\end{proof}

\begin{remark}\label{rem:sign-ambiguity}
  The reader might wonder where the sign ambiguity in the maps on
  Khovanov homology appears. Of course, we have not shown that the map
  of Khovanov spectra associated to a link cobordism is independent of
  the decomposition into elementary cobordisms, so the question is
  somewhat premature. But one possibility is that the ambiguity comes
  from the ambiguity in identifying $C^*(\KhSpace(L))$ with
  $\KhCx(L)$: to make this identification one must orient the cells in
  $C^*(\KhSpace(L))$, and that choice of orientation may be
  unnatural. If so, the map of Khovanov homotopy types could be
  completely well-defined, not just up to sign.
\end{remark}

\subsection{Completion of the proof of \texorpdfstring{\Theorem{coho-op-commute}}{Theorem}}\label{subsec:prove-coho-op-commute}

\begin{proof}[Proof of \Theorem{coho-op-commute}]
  Fix a diagram $D_i$ for $L_i$. We can view $S$ as given by a
  sequence $S_n\circ S_{n-1}\circ\dots\circ S_1$ of Reidemeister and
  Morse moves connecting $D_1$ to $D_2$. The map $F_S$ is the
  composition $F_{S_n}\circ \dots \circ F_{S_1}$, so it suffices
  to prove that Diagram~\eqref{eq:coho-op-commute} commutes (up to sign)
  for a single $S_i$.  If $S_i$ is a Reidemeister move then by
  Propositions~\ref{prop:extra-choices},~\ref{prop:signs}
  and~\ref{prop:Reid-agree} the
  map $\Phi_{S_i}$ of Khovanov homotopy types given
  by~\cite[Theorem~\ref*{KhSp:thm:kh-space}]{RS-khovanov} induces the
  map $\pm F_{S_i}$ on homology. Thus, by naturality of $\alpha$,
  Diagram~\eqref{eq:coho-op-commute} commutes up to sign. If $S_i$ is a Morse
  cobordism then by \Lemma{cup-cap-agree} or \Lemma{saddle-agree}
  there is a map $\Phi_{S_i}$ of Khovanov homotopy types inducing the map
  $F_{S_i}$ on cohomology. So, again, naturality of $\alpha$
  implies that Diagram~\eqref{eq:coho-op-commute} commutes up to sign. This
  completes the proof.
\end{proof}

\begin{proof}[Proof of \ThmCor{steenrod-alg}]
  This is immediate from \Theorem{coho-op-commute}.
\end{proof}

\section{New s-invariants}\label{sec:new-s}

Fix some field $\Field$ and let
$\al\from\wt{H}^*(\cdot;\Field)\to\wt{H}^{*+n}(\cdot;\Field)$ be a
stable cohomology operation for some $n>0$. Recall
from \Definition{full} that we defined an odd integer $q$ to be
\emph{$\alpha$-half-full} if there is a configuration of the form
\[
\xymatrix{
\langle\wt{a}\rangle\ar[r]\ar@{^(->}[d]&\langle\wh{a}\rangle\ar@{<-}[r]\ar@{^(->}[d]&\langle
a\rangle\ar[r]\ar@{^(->}[d]&\langle\ol{a}\rangle\neq 0\ar@<-2ex>@{^(->}[d]\\
\Kh^{-n,q}(K;\Field)\ar[r]^-{\alpha}&\Kh^{0,q}(K;\Field)\ar@{<-}[r]& H_0(\Filt_q;\Field)\ar[r]& H_0(\BNcx;\Field). 
}               
\]
and \emph{$\alpha$-full} if there is a configuration of the form:
\[
\xymatrix{
\langle\wt{a},\wt{b}\rangle\ar[r]\ar@{^(->}[d]&\langle\wh{a},\wh{b}\rangle\ar@{<-}[r]\ar@{^(->}[d]&\langle
a,b\rangle\ar[r]\ar@{^(->}[d]&\langle\ol{a},\ol{b}\ar@{=}[d]\rangle\\
\Kh^{-n,q}(K;\Field)\ar[r]^-{\alpha}&\Kh^{0,q}(K;\Field)\ar@{<-}[r]& H_0(\Filt_q;\Field)\ar[r]& H_0(\BNcx;\Field). 
}               
\]

Clearly, if $q$ is $\al$-full, then it is
$\al$-half-full. Furthermore, it is easy to see that if $q$ if
$\al$-full (\respectively $\al$-half-full), then $q-2$ is $\al$-full
(\respectively $\al$-half-full). So, 
we defined the following knot invariants in \Definition{rpm-spm}:
\begin{align*}
  r_+^\alpha(K)&=\max\set{q\in 2\Z+1}{q\text{ is
      $\alpha$-half-full}}+1 &   r_-^\alpha(K)&=-r_+^\alpha(\ol{K})\\
  s_+^\alpha(K)&=\max\set{q\in 2\Z+1}{q\text{ is
      $\alpha$-full}}+3 &   s_-^\alpha(K)&=-s_+^\alpha(\ol{K}),
\end{align*}
where $\ol{K}$ denotes the mirror of $K$.

Let us study a few properties of these invariants before we delve into
the proof of \Theorem{slice-bound}.

\begin{lem}\label{lem:unknot-s}
 If $U$ denotes the unknot, then $s^{\al}_{\pm}(U)=r^{\al}_{\pm}(U)=0$.
\end{lem}

\begin{proof}
  This is immediate if one considers the zero-crossing diagram of the
  unknot.
\end{proof}

\begin{lem}\label{lem:rs-increase-by-2}
  For any knot $K$,
  $r^{\al}_+(K),s^{\al}_+(K)\in\{s^{\Field}(K),s^{\Field}(K)+2\}$;
  therefore, $r^{\al}_+(K)=s^{\Field}(K)$ if $s^{\Field}(K)+1$ is not
  $\al$-half-full and $r^{\al}_+(K)=s^{\Field}(K)+2$ otherwise; and
  $s^{\al}_+(K)=s^{\Field}(K)$ if $s^{\Field}(K)-1$ is not $\al$-full
  and $s^{\al}_+(K)=s^{\Field}(K)+2$ otherwise.
\end{lem}

\begin{proof}
If $q$ is $\al$-full, then there is a configuration of the form
\[
\xymatrix{
\langle\wt{a},\wt{b}\rangle\ar[r]\ar@{^(->}[d]&\langle\wh{a},\wh{b}\rangle\ar@{<-}[r]\ar@{^(->}[d]&\langle
a,b\rangle\ar[r]\ar@{^(->}[d]&\langle\ol{a},\ol{b}\rangle\ar@{=}[d]\\
\Kh^{-n,q}(K;\Field)\ar[r]^-{\alpha}&\Kh^{0,q}(K;\Field)\ar@{<-}[r]& H_0(\Filt_q;\Field)\ar[r]& H_0(\BNcx;\Field), 
}               
\]
and hence the map $H_0(\Filt_q;\Field)\to H_0(\BNcx;\Field)$ is surjective. Therefore,
$q\leq s^{\Field}_{\min}(K)=s^{\Field}(K)-1$. This implies
\[
s^{\alpha}_+(K)=\max\set{q\in 2\Z+1}{q\text{ is
      $\alpha$--full}}+3\leq s^{\Field}(K)+2.
\]

For the other direction, we need to show that $s^{\Field}(K)-3$ is
$\al$-full. Let $q=s^{\Field}-3$. Choose $a',b'\in
H_0(\Filt_{q+2};\Field)$ so that $a',b'$ maps to some basis
$\ol{a},\ol{b}\in H_0(\BNcx;\Field)$. Let $a,b\in H_0(\Filt_{q};\Field)$ be
the images of $a',b'$ under the map $H_0(\Filt_{q+2};\Field)\to
H_0(\Filt_{q};\Field)$. The exact sequence
\[
H_0(\Filt_{q+2};\Field)\to H_0(\Filt_{q};\Field)\to \Kh^{0,q}(K;\Field)
\]
implies that $a,b$ maps to $0$ in $\Kh^{0,q}(K;\Field)$. Therefore,
there is the following configuration
\[
\xymatrix{
\langle 0\rangle\ar[r]\ar@{^(->}[d]&\langle 0\rangle\ar@{<-}[r]\ar@{^(->}[d]&\langle
a,b\rangle\ar[r]\ar@{^(->}[d]&\langle\ol{a},\ol{b}\rangle\ar@{=}[d]\\
\Kh^{-n,q}(K;\Field)\ar[r]^-{\alpha}&\Kh^{0,q}(K;\Field)\ar@{<-}[r]& H_0(\Filt_{q};\Field)\ar[r]& H_0(\BNcx;\Field) 
}               
\]
and hence $q$ is $\al$-full. Therefore, $s^{\alpha}_+(K)\geq
q+3=s^{\Field}(K)$.

The argument for $r^{\al}_+(K)$ is similar.
\end{proof}

\begin{cor}\label{cor:alpha-zero}
  For any knot $K$, if $\al(\Kh^{-n,s^{\Field}(K)+1}(K;\Field))=0$, then
  $r_+^{\al}(K)=s^{\Field}(K)$; and if
  $\al(\Kh^{-n,s^{\Field}(K)-1}(K;\Field))=0$, then $s_+^{\al}(K)=s^{\Field}(K)$.
\end{cor}

\begin{proof}
  Assume $\al(\Kh^{-n,s^{\Field}(K)-1}(K;\Field))=0$ and
  $s_+^{\al}(K)\neq s^{\Field}(K)$. Therefore by
  \Lemma{rs-increase-by-2}, $s^{\Field}(K)-1$ is $\al$-full. Let
  $q=s^{\Field}(K)-1$. Since the image of $\al$ is zero, there exists
  the following configuration
  \[
  \xymatrix{
    \langle\wt{a},\wt{b}\rangle\ar[r]\ar@{^(->}[d]&\langle 0\rangle\ar@{<-}[r]\ar@{^(->}[d]&\langle
    a,b\rangle\ar[r]\ar@{^(->}[d]&\langle\ol{a},\ol{b}\rangle\ar@{=}[d]\\
    \Kh^{-n,q}(K;\Field)\ar[r]^-{\alpha}&\Kh^{0,q}(K;\Field)\ar@{<-}[r]& H_0(\Filt_{q};\Field)\ar[r]& H_0(\BNcx;\Field). 
  }               
  \]
  The exact sequence
  \[
  H_0(\Filt_{q+2};\Field)\to H_0(\Filt_{q};\Field)\to \Kh^{0,q}(K;\Field)
  \]
  implies that $a,b$ are the images of some elements, say $a',b'\in
  H_0(\Filt_{q+2};\Field)$. Therefore, the map
  $H_0(\Filt_{q+2};\Field)\to H_0(\BNcx;\Field)$ is surjective, which
  contradicts with the statement that $s^{\Field}_{\min}(K)=q$.

  Once again, the argument for $r_+^{\al}(K)$ is similar.
\end{proof}

\begin{cor}
  For any knot $K$,
  $r_-^{\al}(K),s_-^{\al}(K)\in\{s^{\Field}(K),s^{\Field}(K)-2\}$;
  therefore, we have
  $\max\{|r_\pm^{\al}(K)|\},\max\{|s_\pm^{\al}(K)|\}
  \in\{|s^{\Field}(K)|,|s^{\Field}(K)|+2\}$.
\end{cor}

\begin{proof}
  This is immediate from \Lemma{rs-increase-by-2}, the definition of
  $r_-^{\al}$ and $s_-^{\al}$ (\Definition{rpm-spm}) and the fact that
  $s^{\Field}(\ol{K})=-s^{\Field}(K)$ (\Corollary{concord-homo}).
\end{proof}

\begin{proof}[Proof of \Theorem{slice-bound}]
  The proof essentially follows the proof of
  \Corollary{s-bounds-genus}. Let $q=s_+^{\al}(K_1)-3$. Choose elements
  $\wt{a},\wt{b}\in\Kh^{-n,q}(K_1;\Field)$, $\wh{a},\wh{b}\in\Kh^{0,q}(K_1;\Field)$,
  $a,b\in H_0(\Filt_q\BNcx(K_1);\Field)$ and $\ol{a},\ol{b}\in H_0(\BNcx(K_1),\Field)$
  satisfying:
  \[
  \xymatrix{
    \langle\wt{a},\wt{b}\rangle\ar[r]\ar@{^(->}[d]&\langle\wh{a},\wh{b}\rangle\ar@{<-}[r]\ar@{^(->}[d]&\langle
    a,b\rangle\ar[r]\ar@{^(->}[d]&\langle\ol{a},\ol{b}\ar@{=}[d]\rangle\\
    \Kh^{-n,q}(K_1;\Field)\ar[r]^-{\alpha}&\Kh^{0,q}(K_1;\Field)\ar@{<-}[r]& H_0(\Filt_q\BNcx(K_1);\Field)\ar[r]& H_0(\BNcx(K_1);\Field). 
  }               
  \]

  By \Citethm{BN-filt-cob}, the cobordism map
  $F_S\from\BNcx(K_1)\to\BNcx(K_2)$ is a filtered map of filtration
  $\chi(S)=-2g$. By an abuse of notation, let $F_S$ also denote each of
  induced maps $H_0(\BNcx(K_1))\to H_0(\BNcx(K_2))$,
  $H_0(\Filt_q\BNcx(K_1))\to H_0(\Filt_{q-2g}\BNcx(K_2))$ and
  $\Kh^{i,q}(K_1)\to \Kh^{i,q-2g}(K_2)$.

  Since the induced map on the associated graded complex commutes with
  the cohomology operation $\al$ up to a sign (by
  \Theorem{coho-op-commute}), we can pushforward the above
  configuration to get the following configuration:
  \[
  \xymatrix{
    \langle F_S(\wt{a}),F_S(\wt{b})\rangle\ar[r]\ar@{^(->}[d]&\langle F_S(\wh{a}),F_S(\wh{b})\rangle\ar@{<-}[r]\ar@{^(->}[d]&\langle
    F_S(a),F_S(b)\rangle\ar[r]\ar@{^(->}[d]&\langle F_S(\ol{a}),F_S(\ol{b})\ar@{^(->}[d]\rangle\\
    \Kh^{-n,q-2g}(K_2;\Field)\ar[r]^-{\alpha}&\Kh^{0,q-2g}(K_2;\Field)\ar@{<-}[r]& H_0(\Filt_{q-2g}\BNcx(K_2);\Field)\ar[r]& H_0(\BNcx(K_2);\Field). 
  }               
  \]
  Finally, recall that since $\ol{a},\ol{b}$ is a basis for
  $H_0(\BNcx(K_1))$, by \Proposition{S-quasi-iso},
  $F_S(\ol{a}),F_S(\ol{b})$ is a basis for
  $H_0(\BNcx(K_2))$. Therefore, $q-2g$ is $\al$-full for $K_2$ and
  hence,
  \begin{align*}
  s_+^{\al}(K_2)&\geq q-2g+3=s_+^{\al}(K_1)-2g,\text{ or}\\
  2g&\geq s_+^{\al}(K_1)-s_+^{\al}(K_2).
  \end{align*}

  By treating $S$ as a cobordism from $K_2$ to $K_1$, we get $2g\geq
  s_+^{\al}(K_2)-s_+^{\al}(K_1)$, and by combining, we get our desired
  inequality
  \[
  |s_+^{\al}(K_1)-s_+^{\al}(K_2)|\leq 2g.
  \]

  One can take mirrors to get a connected genus $g$ cobordism from
  $\ol{K}_1$ to $\ol{K}_2$. Therefore, we get
  \[
  |s_-^{\al}(K_1)-s_-^{\al}(K_2)|=|s_+^{\al}(\ol{K}_1)-s_+^{\al}(\ol{K}_2)|\leq 2g.
  \]

  Finally, since for the unknot $U$, $s^{\al}_{\pm}(U)=0$ (from
  \Lemma{unknot-s}), we get the desired slice genus bounds. The story
  for $r_{\pm}^{\al}$ is similar.
\end{proof}

\section{Computations}\label{sec:computations}

The first cohomology operation that comes to mind reduces to
Rasmussen's $s$ invariant:
\begin{lemma}
  Suppose $\al$ is the zero map
  $\wt{H}^*(\cdot;\Field)\to\wt{H}^{*+n}(\cdot;\Field)$ for some
  $n>0$. Then $r_{\pm}^{\al}=s_{\pm}^{\al}=s^{\Field}$.
\end{lemma}
\begin{proof}
  This is immediate from \Corollary{alpha-zero}.
\end{proof}

For most of the rest of the section we restrict our attention to the
cohomology operations $\Sq^2\from
\wt{H}^*(\cdot;\F_2)\to\wt{H}^{*+2}(\cdot;\F_2)$ and $\Sq^1\from
\wt{H}^*(\cdot;\F_2)\to\wt{H}^{*+1}(\cdot;\F_2)$. We start with $\Sq^2$.
\begin{lemma}\label{lem:thin-trivial}
  Let $K$ be a knot.
  \begin{enumerate} 
  \item If $\Kh(K;\F_2)$ is supported on two adjacent diagonals, then
    $r_\pm^{\Sq^2}(K) = s_\pm^{\Sq^2}(K) =s^{\F_2}(K)$.
  \item If $\Kh(K;\F_2)$ is supported on three adjacent diagonals,
    then $r_\pm^{\Sq^2}(K)=s^{\F_2}(K)$.
  \end{enumerate}
\end{lemma}
\begin{proof}
  The first statement is obvious from \Corollary{alpha-zero} since the
  operation $\Sq^2\from\Kh^{i,j}\to\Kh^{i+2,j}$ vanishes identically
  for these knots. For the second statement, observe that
  $\Kh(K;\F_2)$ is non-zero on the bigradings $(0,s^{\F_2}\pm 1)$ and
  hence has to be zero on the bigrading $(-2,s^{\F_2}+1)$. Therefore,
  once again via \Corollary{alpha-zero}, we are done.
\end{proof}

\begin{proof}[Proof of \Theorem{sq2-good}]
  Consider the knot $K=9_{42}$. From direct computation, we get that
  $s^{\F_2}(K)=0$; and by direct computation or by comparing with the
  knot tables \cite{KAT-kh-knotatlas}, we learn that the ranks of
  $\Kh(K;\F_2)$ are given by
  \[
  \begin{array}{r|rrrrrrr}
    &-4&-3&-2&-1&\phantom{-}0&\phantom{-}1&\phantom{-}2\\
    \hline
      7&.&.&.&.&.&.&1\\
      5&.&.&.&.&.&1&1\\
      3&.&.&.&.&1&1&.\\
      1&.&.&.&2&2&.&.\\
     -1&.&.&1&2&1&.&.\\
     -3&.&1&1&.&.&.&.\\
     -5&1&1&.&.&.&.&.\\
     -7&1&.&.&.&.&.&.\\
  \end{array}
  \]
  From \cite[Table 1]{RS-steenrod}, we see that $\Sq^2\from
  \Kh^{-2,-1}(K;\F_2)\to \Kh^{0,-1}(K;\F_2)$ is
  surjective.\footnote{The action of $\Sq^2$ on the Khovanov homology
    of knots has been computed independently by C.~Seed (for knots up
    to $14$ crossings) \cite{See-kh-knotkit}.} Furthermore, since $s^{\F_2}(K)=0$, the map
  $H_0(\Filt_{-1};\F_2)\to H_0(\BNcx;\F_2)=\F_2\oplus\F_2$ is also
  surjective. Therefore, $-1$ is $\Sq^2$-full for $K$, and hence by
  \Lemma{rs-increase-by-2}, $s_+^{\Sq^2}(K)=s^{\F_2}(K)+2=2$.
\end{proof}

\begin{remark}
  Recall that $\sigma(9_{42})=2$, so $9_{42}$ is not topologically
  slice. (The knot $9_{42}$ has $g_4=1$, both smoothly and
  topologically.) By contrast, the Ozsv\'ath-Szab\'o concordance
  invariant $\tau$ does vanish for $9_{42}$.
\end{remark}

In \Table{computations}, we present some computations of these new
invariants. Since the Khovanov homologies for all prime knots up to
$11$ crossings are supported on three adjacent diagonals, by
\Lemma{thin-trivial}, $r^{\Sq^2}_{\pm}=s^{\F_2}$; we have only listed
the prime knots up to $11$ crossings for which one of
$s_{\pm}^{\Sq^2}$ differs from $s^{\F_2}$.

We use a Sage program (\url{http://www.sagemath.org/}) to compute
$s^{\Sq^2}_{\pm}$. The main program is main.sage; we extract data from
the Knot Atlas (\cite{KAT-kh-knotatlas}) into the file extracted.sage;
finally we use the program wrapper.sage to run the main program on
these knots. All the programs are available at any of the following
locations:

\url{http://math.columbia.edu/~sucharit/programs/newSinvariants/}

\url{https://github.com/sucharit/newSinvariants}

\tiny
\begin{center}
\begin{longtable}{crrr|crrr|crrr|crrr}
\caption{}\label{table:computations}\\*
\toprule 
$K$&$s^{\F_2}$&$s^{\Sq^2}_+$&$s^{\Sq^2}_-$&
$K$&$s^{\F_2}$&$s^{\Sq^2}_+$&$s^{\Sq^2}_-$&
$K$&$s^{\F_2}$&$s^{\Sq^2}_+$&$s^{\Sq^2}_-$ & $K$&$s^{\F_2}$&$s^{\Sq^2}_+$&$s^{\Sq^2}_-$\\*
\midrule
$9_{42}$ & $0$&$2$&$0$ & $10_{132}$ & $-2$&$0$&$-2$ & $10_{136}$ & $0$&$2$&$0$ & $K11n12$ & $2$&$2$&$0$ \\
$K11n19$ & $-2$&$-2$&$-4$ & $K11n20$ & $0$&$0$&$-2$ & $K11n24$ & $0$&$2$&$0$ & $K11n70$ & $2$&$4$&$2$ \\
$K11n79$ & $0$&$2$&$0$ & $K11n92$ & $0$&$0$&$-2$ & $K11n96$ & $0$&$2$&$0$ & $K11n138$ & $0$&$2$&$0$ \\
\bottomrule
\end{longtable}
\end{center}
\normalsize

Examples are harder to find for $\Sq^1$. Using computations of
C.~Seed's (see also \Remark{cotton-good}), we can show that  
$K14n19265$ is one example.

\begin{proof}[Proof of \Theorem{sq1-good}]
Consider the knot $K=K14n19265$. KnotTheory \cite{kat-kh-knottheory} provides us with $\Kh(K;\Z)$.

\[
\tiny
  \begin{array}{r|cccccccccccccccc}
    &-8&-7&-6&-5&-4&-3&-2&-1&0&1&2&3&4&5&6\\
    \hline
    9&.&.&.&.&.&.&.&.&.&.&.&.&.&.&\Z\\
    7&.&.&.&.&.&.&.&.&.&.&.&.&.&.&\F_{\!2}\\
    5&.&.&.&.&.&.&.&.&.&.&.&.&\Z&\Z&.\\
    3&.&.&.&.&.&.&.&.&.&.&\Z&\Z&\F_{\!2}&.&.\\
    1&.&.&.&.&.&.&.&.&\Z&.&\F_{\!2}&\Z\oplus\F_{\!2}&.&.&.\\
    -1&.&.&.&.&.&.&.&.&\Z\oplus\F_{\!2}^2&\Z^2&\Z&.&.&.&.\\
    -3&.&.&.&.&.&.&\Z^2&\Z\oplus\F_{\!2}&\F_{\!2}^2&\F_{\!2}&.&.&.&.&.\\
    -5&.&.&.&.&.&\Z&\F_{\!2}^3&\F_{\!2}^2&\Z&.&.&.&.&.&.\\
    -7&.&.&.&.&\Z&\Z^2\oplus\F_{\!2}^3&\Z\oplus\F_{\!2}&.&.&.&.&.&.&.&.\\
    -9&.&.&.&\Z^2&\Z\oplus\F_{\!2}^2&\F_{\!2}^2&.&.&.&.&.&.&.&.&.\\
    -11&.&.&\Z&\Z\oplus\F_{\!2}^2&\F_{\!2}&.&.&.&.&.&.&.&.&.&.\\
    -13&.&\Z&\Z^2\oplus\F_{\!2}&.&.&.&.&.&.&.&.&.&.&.&.\\
    -15&.&\Z\oplus\F_{\!2}&.&.&.&.&.&.&.&.&.&.&.&.&.\\
    -17&\Z&.&.&.&.&.&.&.&.&.&.&.&.&.&.\\
  \end{array}
\]
\normalsize

Direct computations done by C.~Seed (see \Remark{cotton-good}) tell
us that $s^{\F_2}(K)=-2$, while $s^{\QQ}(K)=0$ is forced by the
form of $\Kh(K;\QQ)$. We want to show that $s^{\Sq^1}(K)=0$ as
well; using \Lemma{rs-increase-by-2}, we only need to show that $-3$
is $\Sq^1$-full.

The quantum filtration on $\BNcx(K)/\Filt_{-3}\BNcx(K)$ leads to a spectral
sequence which starts at $\bigoplus_{q<-3}\Kh^{*,q}(K;\Z)$, 
and converges to $H_*(\BNcx(K)/\Filt_{-3}\BNcx(K);\Z)$,
and whose differentials increase the homological grading by $1$ and
increase the quantum grading by at least $2$. Therefore, from the form
of $\Kh(K;\Z)$, we can conclude that
$H_0(\BNcx(K)/\Filt_{-3}\BNcx(K);\Z)=\Z$ and
$H_i(\BNcx(K)/\Filt_{-3}\BNcx(K);\Z)=0$ for all $i>0$.

The short exact sequence
\[
0\to\Filt_{-3}\BNcx(K)\stackrel{\iota}{\longrightarrow}\BNcx(K)\stackrel{\pi}{\longrightarrow}\BNcx(K)/\Filt_{-3}\BNcx(K)\to 0
\]
furnishes us with a long exact sequence
\[
\cdots\to
H_0(\Filt_{-3}\BNcx(K);R)\stackrel{\iota^R_0}{\longrightarrow}
H_0(\BNcx(K);R)\cong R^2\stackrel{\pi^R_0}{\longrightarrow}
H_0(\BNcx(K)/\Filt_{-3}\BNcx(K);R)\cong R\to\cdots
\]
over any ring $R$.

Since $s^{\QQ}_{\min}(K)=-1\geq -3$, the map $\iota_0^{\QQ}$ is
surjective, and therefore the map $\pi_0^{\QQ}\from\QQ^2\to\QQ$ is
zero. Lack of torsion implies that the map $\pi_0^{\Z}\from\Z^2\to\Z$
is zero as well, and hence the map $\iota_0^{\Z}$ is surjective.

Observe that since $H_*(\BNcx(K);\Z)$ is torsion-free, the map
$H_0(\BNcx(K);\Z)\to H_0(\BNcx(K);\F_2)$ is surjective. Choose a basis
$\ol{a},\ol{b}$ of $H_0(\BNcx(K);\F_2)$, and choose elements
$\ol{\al},\ol{\be}\in H_0(\BNcx(K);\Z)$ that map to
$\ol{a},\ol{b}$. Then choose $\al,\be\in H_0(\Filt_{-3}\BNcx(K);\Z)$
which map $\ol{\al},\ol{\be}$, and consider the
following induced configuration:
\[
\tiny
\xymatrix{
&\langle\wh{a},\wh{b}\rangle\ar@{<-}[rr]\ar@{^(->}'[d][dd]&&\langle
a,b\rangle\ar[rr]\ar@{^(->}'[d][dd]&&\langle\ol{a},\ol{b}\ar@{=}[dd]\rangle\\
\langle\wh{\al},\wh{\be}\rangle\ar@{<-}[rr]\ar@{^(->}[dd]\ar[ru]&&\langle
\al,\be\rangle\ar[rr]\ar@{^(->}[dd]\ar[ru]&&\langle\ol{\al},\ol{\be}\ar@{^(->}[dd]\rangle\ar[ru]&\\
&\Kh^{0,-3}(K;\F_2)\ar@{<-}'[r][rr]&& H_0(\Filt_{-3};\F_2)\ar'[r][rr]&&
H_0(\BNcx;\F_2)\\
\Kh^{0,-3}(K;\Z)\ar@{<-}[rr]\ar[ru]&& H_0(\Filt_{-3};\ZZ)\ar[rr]\ar[ru]&& H_0(\BNcx;\ZZ).\ar[ru]& 
}               
\]
\normalsize

Since $\wh{a},\wh{b}\in\Kh^{0,-3}(K;\F_2)$ admit integral
lifts, $\Sq^1(\wh{a})=\Sq^1(\wh{b})=0$. From the form of $\Kh(K;\Z)$,
we know that the following is exact
\[
\Kh^{-1,-3}(K;\F_2)\stackrel{\Sq^1}{\longrightarrow}\Kh^{0,-3}(K;\F_2)\stackrel{\Sq^1}{\longrightarrow}\Kh^{1,3}(K;\F_2)
\]
(the rank of the first map is $2$, and the rank of the second map is
$1$, and $\Kh^{0,-3}(K;\F_2)\cong\F_2^3$). Therefore, $\wh{a}$ and $\wh{b}$ must lie in the image of $\Sq^1$ as
well. Thus we have a configuration
\[
\xymatrix{
\langle\wt{a},\wt{b}\rangle\ar[r]\ar@{^(->}[d]&\langle\wh{a},\wh{b}\rangle\ar@{<-}[r]\ar@{^(->}[d]&\langle
a,b\rangle\ar[r]\ar@{^(->}[d]&\langle\ol{a},\ol{b}\ar@{=}[d]\rangle\\
\Kh^{-1,3}(K;\F_2)\ar[r]^-{\Sq^1}&\Kh^{0,-3}(K;\F_2)\ar@{<-}[r]& H_0(\Filt_{-3};\F_2)\ar[r]& H_0(\BNcx;\F_2),
}               
\]
thereby establishing $-3$ is $\Sq^1$-full, and hence $s^{\Sq^1}(K)=0$.
\end{proof}

\section{Further remarks}\label{sec:further}
For convenience, throughout this paper we have used Khovanov homology
with coefficients in a field. It is natural to wonder how the $s$
invariants depend on the field. In a previous version of this paper,
we asked:
\begin{question}\label{ques:s-field}
  Let $\Field$ and $\Field'$ be fields. Is there a knot $K$ so that
  $s^{\Field}(K)\neq s^{\Field'}(K)$?
\end{question}

\begin{remark}\label{rem:cotton-good}
  It was claimed in~\cite[Theorem 4.2]{MTV-Kh-s-invts} that
  $s^{\Field}$ is independent of $\Field$,
  but, as noted earlier, there is
  a gap in the proof of~\cite[Proposition 3.2]{MTV-Kh-s-invts}.  Since
  the first draft of this paper, using his package knotkit
  \cite{See-kh-knotkit}, Cotton Seed has found examples of knots,
  $K14n19265$ being one of them, where $s^{\F_2}\neq s^{\QQ}$.
\end{remark}

A more quantitative version of \Question{s-field} is the following:
\begin{question}\label{ques:s-conc}
  Let $\mathscr{C}$ denote the smooth concordance group and let
  $\mathscr{T}\sbs\mathscr{C}$ denote the subgroup generated by the
  topologically slice knots.  Consider the homomorphism
  $s\defeq(s^{\QQ},s^{\F_2},s^{\F_3},s^{\F_5},\ldots)\from
  \mathscr{C}\to\Z^{\omega}$.  What are the images of $\frac{s}{2}$ and
  $\frac{\restrict{s}{\mathscr{T}}}{2}$?
\end{question}

If one considered Khovanov homology with
coefficients in $\ZZ$, there are many possible variants of
$s$. Specifically, for $m\in \ZZ$, we can consider the invariants
\begin{align*}
  s^{\ZZ,m}_{\min}(K)&=  \max\set{q\in2\Z+1}{\ZZ/m\text{ surjects onto }H_*(\BNcx(K;\ZZ))/i_*H_*(\Filt_q\BNcx(K;\ZZ))}\\
  s^{\ZZ,m}_{\max}(K)&= \max\set{q\in 2\Z+1}{\Z\oplus\Z/m\text{
      surjects onto }H_*(\BNcx(K;\ZZ))/i_*H_*(\Filt_q\BNcx(K;\ZZ))}.
\end{align*}
It is straightforward to verify that $s^{\ZZ,m}_{\max}-1$ and
$s^{\Z,m}_{\min}+1$ give concordance invariants leading to slice
genus bounds.  Along the lines of Questions~\ref{ques:s-field}
and~\ref{ques:s-conc}:
\begin{question}
  For different $m\in\Z$, how are the invariants $s^{\ZZ,m}_{\max}-1$
  (\respectively $s^{\Z,m}_{\min}+1$) related?
\end{question}

In this context, one can use cohomology operations over $\ZZ$,
similarly to \Definition{full}, to obtain other possibly new
concordance invariants.  One could go further and define more
invariants by counting more complicated configurations using
cohomology operations. For example, one could define $q$ to be
\emph{$\Sq^2$-$\Sq^1$-$\Sq^2$-half-full} if there are elements
$\overline{a}\in H_0(C;\Field_2)$, $a\in H_0(\Filt_q;\Field)$,
$\wh{a}\in \Kh^{0,q}(K;\Field_2)$, $\wt{a}_1\in
\Kh^{-2,q}(K;\Field_2)$, $\wt{a}_2\in \Kh^{-1,q}(K;\Field_2)$, and
$\wt{a}_3\in \Kh^{-3,q}(K;\Field_2)$ such that 
\begin{align*}
  \overline{a}&\neq0 &
  i_*(a)&=\overline{a} &
  p_*(a)&=\wh{a} \\
  \Sq^2(\wt{a}_1)&=\wh{a} &
  \Sq^1(\wt{a}_1)&=\wt{a}_2&
  \Sq^2(\wt{a}_3)&=\wt{a}_2,
\end{align*}
and use this notion to define an invariant
$r^{\Sq^2,\Sq^1,\Sq^2}$. (Here, $p\co \Filt_q\to \Filt_q/\Filt_{q+2}$
denotes the quotient map.) That is, $r^{\Sq^2,\Sq^1,\Sq^2}$ is defined
by looking for configurations of the form
{\tiny
\[
\xymatrix{
\langle \wt{a}_3\rangle\ar@{^(->}[d]\ar[r]& \langle \wt{a}_2\rangle\ar@{^(->}[d] & \langle\wt{a}_1\rangle\ar[l]\ar[r]\ar@{^(->}[d]&\langle\wh{a}\rangle\ar@{<-}[r]\ar@{^(->}[d]&\langle
a\rangle\ar[r]\ar@{^(->}[d]&\langle\ol{a}\rangle\neq
0\ar@<-2ex>@{^(->}[d]\\
\Kh^{-3,q}(K;\Field_2)\ar[r]^-{\Sq^2}& \Kh^{-1,q}(K;\Field_2) & 
\Kh^{-2,q}(K;\Field_2)\ar[r]^-{\Sq^2}\ar[l]_-{\Sq^1}&\Kh^{0,q}(K;\Field_2)\ar@{<-}[r]^-{p_*}& H_0(\Filt_q;\Field_2)\ar[r]^-{i_*}& H_0(\BNcx;\Field_2). 
}
\]}

We have shown that $s^{\Sq^2}_+$ is distinct from $s^{\F_2}$.  It is natural to ask how many more of these invariants are
new:
\begin{question}
  For which $\alpha$'s are the invariants $r^\alpha_{\pm}$ (\respectively
  $s^\alpha_{\pm}$) distinct? More generally, for which configurations are
  the resulting concordance invariants different?
\end{question}

Finally, in light of~\cite{FGMW-kh-exotic} and~\cite{KM-Kh-Ras-not-exotic}, it
is natural to ask:
\begin{question}
  Does $s_\pm^{\alpha}$ or $r_\pm^{\alpha}$ give genus bounds for
  bounding surfaces in homotopy $4$-balls?
\end{question}

\bibliography{Rasmus}

\end{document}